\documentclass[a4paper,11pt]{amsart}

\usepackage{amsmath,amssymb,amsthm,mathtools,tikz,comment,bbm,cprotect}
\usepackage[margin=2.5cm]{geometry}
\usepackage{thm-restate} 
\usepackage{enumerate}
\usepackage[shortlabels]{enumitem}
\usepackage{hyperref}
\usepackage[noadjust]{cite}

\theoremstyle{plain}
\newtheorem{thm}{Theorem}[section]
\newtheorem{prop}[thm]{Proposition}
\newtheorem{cor}[thm]{Corollary}

\newtheorem{lem}[thm]{Lemma}

\theoremstyle{definition}
\newtheorem{defn}[thm]{Definition}
\newtheorem{ex}[thm]{Example}
\newtheorem{rem}[thm]{Remark}

\newtheorem*{remark*}{Remark}
\newtheorem*{thm*}{Theorem}
\newtheorem*{defn*}{Definition}
\newtheorem*{lem*}{Lemma}
\newtheorem*{cor*}{Corollary}

\newcommand{\N}{\mathbb{N}}
\newcommand{\Z}{\mathbb{Z}}
\newcommand{\Q}{\mathbb{Q}}
\newcommand{\R}{\mathbb{R}}

\newcommand{\F}{\mathbb{F}}

\newcommand{\cosets}{\mathcal{C}}

\newcommand{\cosetsfin}{\cosets_{\textup{fin}}}
\newcommand{\cosetsinf}{\cosets_{\textup{inf}}}
\newcommand{\G}[2][G]{X_{#2}(#1)}
\newcommand{\Gfin}{\G{\textup{fin}}}
\newcommand{\sgpfin}{\leq_{\textup{fin}}}
\newcommand{\sgpnfin}{\unlhd_{\textup{fin}}}

\newcommand{\dc}[1][]{\operatorname{dc}_{#1}}
\newcommand{\dcRF}{\dc[\textup{RF}]}

\newcommand{\dcCpt}{\dc[\textup{Cpt}]}
\newcommand{\dcRAm}{\dc[\textup{RAm}]}
\newcommand{\dcCCM}{\dc[\textup{CCM}]}

\DeclareMathOperator{\Aut}{Aut}
\DeclareMathOperator{\Comm}{Comm}
\DeclareMathOperator{\Core}{Core}
\DeclareMathOperator{\res}{Res}

\newcommand{\Pset}{\mathcal{P}}
\newcommand{\Pfin}{\Pset_{\textup{fin}}}
\DeclareMathOperator{\cc}{cc}

\setenumerate{label=\textup{(\roman*)}} 

\title[Coset correct means and degrees of commutativity]{Coset correct means on groups and the probability that two elements commute}
\author{Armando Martino}
\address[A. Martino]{Mathematical Sciences, University of Southampton, University Road, Southampton SO17 1BJ, United Kingdom}
\email{A.Martino@soton.ac.uk}
\author{Motiejus Valiunas}
\address[M. Valiunas]{Instytut Matematyczny, Uniwersytet Wroc{\l}awski, plac Grunwaldzki 2, 50-384 Wroc{\l}aw, Poland}
\email{motiejus.valiunas@math.uni.wroc.pl}

\newcommand{\ThmDCCCM}{%
Let $G$ be a group equipped with a CCM $\mu$. Then $\dc[\mu](G) > 0$ if and only if $G$ is finite-by-abelian-by-finite.  Moreover, $\dcCCM(G)\coloneqq\dc[\mu](G)$ is a rational number that does not depend on the choice of~$\mu$.%
}

\newcommand{\ThmDCRF}{%
Let $G$ be a residually finite group with profinite completion $\widehat{G}$, and let $\widehat\mu$ be the Haar measure on $\widehat{G}^2$.  Then we have
\[
\widehat\mu(\Comm(\widehat{G})) = \widehat\mu(\overline{\Comm(G)}) = \dcRF(G) = \dcCCM(G).
\]
Moreover, this number is positive if and only if $G$ is virtually abelian.%
}

\newcommand{\ThmDCAm}{%
Let $G$ be an amenable group, and let $\mu$ be a finitely additive left invariant probability mean on $G^2$.  Then $\mu(\Comm(G)) > 0$ if and only if $G$ is finite-by-abelian-by-finite.%
}

\newcommand{\ThmDefect}{%
Let $G$ be a group. Then the defect $\delta_G$ of $G$ is equal to $\delta_\ell(\mu)$ for some finitely additive probability mean $\mu$ on $G$. Moreover, $\delta_G$ equals either $0$ or $1$, and is $0$ precisely when $G$ is amenable.%
}

\begin{document}

\begin{abstract}
Amenable groups are those admitting an \emph{invariant mean}---a finitely additive probability mean that assigns equal ``weight'' to any two translates of the same set. 

We introduce \emph{coset correct means} (\emph{CCMs}), a class of finitely additive means that, for any subgroup, assigns equal weight to all its cosets, weakening and therefore generalising the notion of an invariant mean.  We show that, unlike the case for invariant means, \emph{every} group admits a CCM and give two constructions---one via the Ultrafilter Lemma and one via the Hahn--Banach Theorem---both relying on a Theorem of B.~H.\ Neumann.
Using CCMs, we define a degree of commutativity for arbitrary groups, measuring the ``probability'' that two random elements of a group commute. We prove that this degree of commutativity is independent of the choice of CCM and is positive precisely for finite‑by‑abelian‑by‑finite groups, recovering and unifying previous characterisations. 

We also introduce a defect function that quantifies the failure of left invariance for finitely additive means, and define the defect of a group as the infimum of these. We then prove a dichotomy: the defect for a group is either $0$ or $1$, with $0$ characterising amenable groups.
\end{abstract}

\maketitle

\tableofcontents

\section{Introduction}

The goal of this paper is to define finitely additive probability means for any group which give the ``correct'' answer for subgroups and their cosets, namely the reciprocal of the index of the subgroup. We call such a mean a coset correct mean, a CCM. We then use these means to define a probability that two elements of a group commute---the degree of commutativity---and thence deduce structural results about groups where this probability is positive. Such means are already seen to exist for amenable groups but this is well known to exclude many groups. 

As a starting point for thinking about such means it is perhaps instructive to consider the following well-known Theorem:
\begin{center}
    $\R^n$ is not the union of finitely many affine hyperplanes.
\end{center}
There are various proofs of this fact but one of the more straightforward ways to verify this is to consider the Lebesgue measure; the union of finitely many affine hyperplanes has measure~$0$ and so cannot equal Euclidean space. This is straightforward up to constructing the Lebesgue measure, of course. 

After some work one can verify the analogous result holds for any infinite field but the most general version of this type of result is due to B.~H.\ Neumann. 

\begin{thm}[Neumann {\cite[Lemma~4.5]{Neumann1954}}] \label{thm:neumann}
    Suppose that $G = \bigcup_{i=1}^k g_iH_i$ for a group $G$, elements $g_1,\ldots,g_k \in G$ and subgroups $H_1,\ldots,H_k \leq G$.  Then we have
    \[
    \sum_{i=1}^k \frac{1}{[G:H_i]} \geq 1,
    \]
    where the reciprocal $\frac{1}{[G:H_i]}$ is understood to be zero whenever $H_i$ has infinite index in $G$.  
    
    In particular, no coset of a finite index subgroup of $G$ can be covered by finitely many cosets of infinite index subgroups.
\end{thm}

Our contribution is to show that this result actually allows one to construct a CCM, a mean on $G$ whereby all subsets are given some non-negative mean and for which all cosets of a given subgroup are given the same mean. Thus, just as in the case of affine hyperplanes in Euclidean space, no group can be realised as a finite union of cosets of infinite index subgroups because the mean of the latter is always zero. We give two constructions of our CCMs, one using the Ultrafilter Lemma and one using the Hahn--Banach Theorem, but both of these constructions rely on the result of B.~H.~Neumann above to verify consistency before appealing to some version of the axiom of choice to extend the mean to all subsets of the group.

In order to motivate the definition of a coset correct mean we recall the definition of an amenable group as a group admitting a finitely additive probability mean. Equivalently, an amenable group is one which does not admit a paradoxical decomposition in the sense of the Banach--Tarski paradox. We will largely take the former point of view, although there are many equivalent definitions. Precisely, 

\begin{defn*}
A group $G$ is said to be \emph{amenable} if there exists a function, $\mu\colon \Pset(G) \to [0,1]$, satisfying the following: 
\begin{enumerate}
    \item $\mu(G) =1$;
    \item $\mu(A \sqcup B) = \mu(A) + \mu(B)$, for all disjoint subsets $A,B \subseteq G$;
    \item \label{it:am-inv} $\mu(gA) = \mu(A)$, for all $A \subseteq G$ and all $g \in G$. 
\end{enumerate}
\end{defn*}

The function $\mu$ is called an \emph{invariant mean} on $G$. It is then clear that:

\begin{cor*}
If $G$ is amenable and $\mu$ is an invariant mean on $G$ then for any subgroup $H$ of $G$ and any $g \in G$ we have 
\[
\mu(gH) = \frac{1}{[G:H]}, 
\]
where the right hand side is understood to be zero when $H$ has infinite index. 
\end{cor*}

We recall that not all groups are amenable. In particular, any group containing a non-abelian free group is not amenable.

The definition of a coset correct mean (CCM) weakens condition \ref{it:am-inv} in the Definition above and replaces it with the statement of the Corollary. Equivalently, the definition of a CCM simply takes the Definition above of amenability and only asks that left invariance is enjoyed by subgroups (and their cosets). As a minor note, left invariance on subgroups is equivalent to right invariance on subgroups, see Remark~\ref{rem:left equals right}.

\begin{defn}[see Definition~\ref{def:ccm}]
Let $G$ be a group. A \emph{coset correct mean} (\emph{CCM}) on $G$ is a function $\mu \colon \Pset(G) \to [0,1]$ satisfying the following properties: 
\begin{enumerate}
    \item $\mu(G) =1$ (\emph{probability mean});
    \item if $A, B \subseteq G $ are disjoint, then $\mu(A \sqcup B) = \mu(A) + \mu(B)$ (\emph{finitely additive});
    \item $\mu(gH) = \frac{1}{[G:H]}$ for all subgroups $H \leq G$ and all $g \in G$ (\emph{left invariant on subgroups)}.
\end{enumerate}
As always, $\frac{1}{[G:H]}$ is understood to be zero when $H$ has infinite index.
\end{defn} 

We give three constructions of coset correct means. The first (in Section~\ref{sec:CCMs}), via ultrafilters and Neumann’s theorem, is combinatorial and forms the basis of our main arguments. The second (in Appendix~\ref{app:Hahn-Banach}), via the Hahn–Banach theorem, places the construction in a functional-analytic framework. The third (in Appendix~\ref{app:random-walks}), using random walks, shows how CCMs arise as limits of genuine probability measures and allows us to reduce the general case to finitely generated groups. While these constructions yield the same objects, they highlight different structural aspects of CCMs.

\begin{thm}[see Theorems~\ref{thm: existence of CCMs},~\ref{thm:HB implies CCM}~and~\ref{thm:CCMsviawalks}] \label{thm: existence of CCMs-intro}
    Every group admits a CCM.
\end{thm}

As an application of this construction, we use this to measure the \textit{degree of commutativity} of a group $G$. That is, we consider the set, $\Comm(G) \coloneqq \{ (g,h) \in G^2 : gh=hg \}$, and measure the size of this set in $G^2$. 

For example, for a finite group $G$, we would take the uniform measure---count elements, in other words---and define $\dc(G) = \frac{\left|\Comm(G) \right|}{|G^2|}$.  For a residually finite group $G$, we could measure the degree of commutativity by looking at increasingly large finite quotients and define $\dcRF(G)$ as the infimum of $\dc(G/N)$ over all finite quotients $G/N$---see Definition~\ref{defn:dc for RF}. 

In the case of a group $G$ equipped with a CCM $\mu$, we first define the product mean $\mu^{(2)}$ on $G^2$, see Definition~\ref{defn: product mean}. From there we define the degree of commutativity of $G$ with respect to $\mu$ as $\dc[\mu](G) = \mu^{(2)}(\Comm(G))$: see Definition~\ref{defn: dc using CCMs}. 

Our results show that defining the degree of commutativity in various senses, this number is positive if and only if the group in question is finite-by-abelian-by-finite, or FAF. We recall that, given families of groups $\mathcal{A}$ and $\mathcal{B}$, a group $G$ is said to be $\mathcal{A}$-by-$\mathcal{B}$ if $G$ admits a normal subgroup $N$ so that $N$ is in $\mathcal{A}$ and $G/N$ is in $\mathcal{B}$. For instance, using the core construction one can show that a group is virtually abelian (admits a finite index abelian subgroup) if and only if it is abelian-by-finite. We note that the apparent ambiguity in calling a group FAF is genuinely only apparent as the class of (finite-by-abelian)-by-finite groups equals the class of finite-by-(abelian-by-finite) groups. More pithily, (FA)F = FAF = F(AF): see Lemma~\ref{lem:faf-equivalent}. 

We then characterise the groups having positive degree of commutativity in the following Theorems. 

\begin{thm}[see Theorem~\ref{thm:dc-CCM-nonzero}] \label{thm:dc-CCM-nonzero-intro}
    \ThmDCCCM
\end{thm} 

In the case of a residually finite group we note that one can use the Haar measure on the profinite completion and we show that this also leads to the same result, on noting that a residually finite FAF group is AF, i.e.\ virtually abelian. 

\begin{thm}[see Theorem~\ref{thm:dc-res-finite}]
    \label{thm:dc-res-finite-intro}
    \ThmDCRF
\end{thm} 

We also look at this degree of commutativity when $G$ is an amenable group and taking an invariant mean (not necessarily the product mean) on the square. In this setting we get:

\begin{thm}[see Theorem~\ref{thm:dc-amenable}] \label{thm:dc-amenable-intro}
    \ThmDCAm
\end{thm}

Finally, since we are considering finitely additive probability means on groups which are invariant on some family of subsets it makes sense to define the \emph{left defect} of such a probability mean $\mu$ on a group $G$ as 
\[ \delta_\ell(\mu) = \sup \{ |\mu(gA)-\mu(A)| : g \in G, A \subseteq G \} \in [0,1].\]
We then define the (\emph{left}) \emph{defect} of $G$ as \[\delta_G \coloneqq \inf_{\mu} \delta_\ell(\mu),\] as in Definition~\ref{defn:defect}. We then get:

\begin{thm}[see Theorem~\ref{thm:defect}]
    \label{thm:defect-intro}
    \ThmDefect
\end{thm} 

Therefore, while the defect provides a natural way to quantify the failure of (left) invariance in finitely additive means, its behaviour turns out to be unexpectedly rigid as it is always either $0$ or $1$. This dichotomy mirrors the amenability of the group (defect zero if and only if the group is amenable) and thus, while conceptually clean, the defect offers no finer gradation among non-amenable groups. In this sense, it is both striking and somewhat limited.

\subsection*{Results in the literature}

Amenable groups were introduced by John von Neumann in 1920s \cite{Neumann1929}, in order to understand the group-theoretic nature of the Banach--Tarski Paradox.  Since then, amenability has been widely studied: see \cite{Bartholdi2018} and references therein.  Throughout the years, many equivalent definitions of amenability have been introduced; the most relevant to us are the existence of invariant means---closely related to the CCMs studied here---and the F{\o}lner condition (Proposition~\ref{prop:folner condition}), related to our construction of CCMs in Section~\ref{sec:CCMs}.

Even though we deal with discrete groups in this paper, it is worth mentioning that such concepts make sense, more generally, for locally compact Hausdorff topological groups.  Such groups $G$ admit Borel measures, called \emph{Haar measures}, that are left-invariant and countably additive, see \cite{DiestelSpalsbury2014} and references therein; however, the total Haar measure for the whole $G$ is infinite unless $G$ is compact.  A notion of amenability for (locally compact Hausdorff) topological groups has also been extensively studied \cite{Willis2003}.

The degree of commutativity for finite groups was first introduced by Erd\H{o}s and Tur\'an in 1960s \cite{ErdosTuran1968}, and has been widely studied since then \cite{Gallagher1970,Gustafson1973,Rusin1979,Neumann1989}.  This has been later generalised to profinite groups by L\'evai and Pyber \cite{LevaiPyber2000} and, more generally, to compact groups by Hofmann and Russo \cite{HofmannRusso2012}.  In there, the degree of commutativity $\dcCpt(G)$ of a compact group $\mu$ was defined analogously to our definition of $\dc[\mu](G)$, by taking $\mu$ to be the Haar measure on $G$.  It was shown for any compact $G$, the number $\dcCpt(G)$ is rational, and positive if and only if $G$ has an open abelian subgroup; this can be therefore viewed as a topological version of Theorem~\ref{thm:dc-CCM-nonzero-intro} (or Theorem~\ref{thm:dc-amenable-intro}) for compact groups.

In \cite{AntolinMartinoVentura2017}, Antol\'in, Ventura and the first author introduced the study of the degree of commutativity $\dc[S](G)$ of a finitely generated group $G = \langle S \rangle$, by computing the proportion of commuting elements in a ball $B_{G,S}(n)$ of radius $n$ in the Cayley graph $\operatorname{Cay}(G,S)$, and taking the limit (superior) as $n \to \infty$. They showed that $\dc[S](G)$ is zero for non-elementary hyperbolic $G$, and similar results were later obtained by the second author for non-(virtually abelian) groups $G$ that are graph products \cite{Valiunas2019rational} and higher rank solvable Baumslag--Solitar groups \cite{Valiunas2019negligibity}.  The main obstacle in these cases, however, was that for many groups, the ball counting measure does not measure the index of subgroups correctly.

In order to avoid such problems, Tointon \cite{Tointon2020} studied sequences of measures that measure indices of subgroups correctly (in the limit), as well as the degree of commutativity for such sequences. In fact, this notion is equivalent to our notion of a CCM, as observed in Remark~\ref{rem:sequenceareCCM}. These measures  can  be obtained from random walks (in a finitely generated group), as well as F{\o}lner sequences (in a countable amenable group); in particular, a version of Theorem~\ref{thm: existence of CCMs-intro} for finitely generated groups follows from the results in \cite{Tointon2020}.  A version of Theorem~\ref{thm:dc-CCM-nonzero-intro} was also proved there, and this has been later extended by Tointon, Ventura and the authors \cite{MTVV} to the \emph{degree of $k$-step nilpotence} of a group $G$---the probability that the iterated commutator $[x_0,x_1,\ldots,x_k]$ is trivial for ``random'' elements $x_0,\ldots,x_k \in G$.

\subsection*{Structure of the paper}

In Section~\ref{sec:prelims} we give some preliminaries on Boolean rings, filters, ultrafilters and ultralimits, and basic structural results for finite-by-abelian-by-finite groups.  We prove Theorem~\ref{thm: existence of CCMs-intro} in Section~\ref{sec:CCMs} using the Ultrafilter Lemma, and give two alternative proofs in the Appendices.  We prove Theorem~\ref{thm:dc-CCM-nonzero-intro} in Sections~\ref{sec:dc>0=>FAF} and~\ref{sec:FAF=>dc-well-defined}, Theorem~\ref{thm:dc-res-finite-intro} in Section~\ref{sec:RF}, Theorem~\ref{thm:dc-amenable-intro} in Section~\ref{sec:Am}, and Theorem~\ref{thm:defect-intro} in Section~\ref{sec:defect}.  Finally, in Section~\ref{sec:conj-classes} we discuss an alternative generalisation of the degree of commutativity to infinite groups, by measuring sets of conjugacy class representatives.

\subsection*{Acknowledgement}

The work was partially supported by the National Science Centre (Poland) grant No.\ 2022/47/D/ST1/00779.

\section{Preliminaries}
\label{sec:prelims}

\subsection*{Boolean rings and filters}

Recall that a \textit{Boolean ring} is a ring in which $x^2=x$  for every element $x$ of the ring. Note that any Boolean ring is commutative and has characteristic 2. Our Boolean rings will always be unital. 

The main examples of Boolean rings arise as powersets. For any set $X$, the powerset of $X$, $\Pset(X)$, is a Boolean ring where the addition is symmetric difference and the multiplication is intersection: for $A , B \subseteq X$, the operations are
\begin{align*}
    A \oplus B &\coloneqq (A \cup B) - (A \cap B), \\
    AB &\coloneqq  A \cap B.
\end{align*}

It is straightforward to see that this defines a ring structure on noting the bijection between $\Pset(X)$ and $(\Z/2\Z)^X$; the operations in $\Pset(X)$ correspond to pointwise addition and multiplication in $(\Z/2\Z)^X$. The set $X$ is the multiplicative identity of the ring and $\varnothing$ is the additive identity. 

It is also worth noting that a Boolean ring is equivalent to a Boolean algebra, meaning that union, complementation and set difference can be expressed in terms of the ring operations:
\begin{align*}
    A \cup B & = A \oplus B \oplus AB, \\
    A^c & = 1 \oplus A, \\
    A - B & = A \oplus AB.
\end{align*}
Hence any (unital) subring of $\Pset(X)$ is also closed under finite unions, complementation and set difference. 

\begin{defn}
    Let \( X \) be a nonempty set. A \emph{filter} on \( X \) is a collection \( \mathcal{F} \subseteq \Pset(X) \) of subsets of \( X \) such that:

\begin{enumerate}
    \item \( \varnothing \notin \mathcal{F} \), \( X \in \mathcal{F} \);
    \item if \( A, B \in \mathcal{F} \), then \( A \cap B \in \mathcal{F} \);
    \item if \( A \in \mathcal{F} \) and \( A \subseteq B \subseteq X \), then \( B \in \mathcal{F} \).
\end{enumerate}

\medskip

A filter \( \mathcal{F} \) on \( X \) is called \emph{principal} if there exists a subset \( A \subseteq X \) such that
\[
\mathcal{F} = \{ B \subseteq X : A \subseteq B \}.
\]
In this case, \( \mathcal{F} \) is said to be \emph{generated} by \( A \), and we may write \( \mathcal{F} = \langle A \rangle \).

\medskip

A filter \( \mathcal{U} \) on \( X \) is called an \emph{ultrafilter} if it is a maximal proper filter on \( X \), that is, there is no filter \(\mathcal{F}'\) on \(X\) satisfying \( \mathcal{F}' \supsetneq \mathcal{U} \). Equivalently, a filter \( \mathcal{U} \) is an ultrafilter if for every subset \( A \subseteq X \), either \( A \in \mathcal{U} \) or \( X - A \in \mathcal{U} \), but not both.
\end{defn}

\begin{rem}
    Note that an ultrafilter is non-principal if and only if it contains no finite subsets of $X$. Hence all ultrafilters on finite sets are principal. A principal ultrafilter is always generated by a singleton. 
    
    Given any infinite set $X$, the Fr\'echet filter on $X$ is the collection of all co-finite sets (sets with finite complement). An ultrafilter on $X$ is non-principal if and only if it contains the Fr\'echet filter. 
\end{rem}

The following is standard. 

\begin{lem}
    Let $X$ be a set and $\mathcal{F} \subseteq \Pset(X)$. Define $\mathcal{F}^c \coloneqq \{ A \subseteq X : A^c \in \mathcal{F}\} $. Then, 
    \begin{enumerate}[label=\textup{(\roman*)}]
        \item $\mathcal{F}$ is a filter if and only if $\mathcal{F}^c$ is a proper ideal. 
        \item $\mathcal{F}$ is a principal filter if and only if $\mathcal{F}^c$ is a principal ideal (generated by a single element).
        \item $\mathcal{F}$ is an ultrafilter if and only if $\mathcal{F}^c$ is a maximal ideal.
    \end{enumerate}
\end{lem}

\begin{rem}
    Since $\Pset(X)$ is a Boolean ring, any quotient of $\Pset(X)$ is also a Boolean ring. In particular, any quotient which is a field satisfies $x^2=x$ for every element and so must be the field with two elements. Therefore, any maximal ideal has index~$2$.
\end{rem}

The existence of (non-principal) ultrafilters require a weak version of the Axiom of Choice. For those not worried about relative strengths, the following is an easy consequence of Zorn's Lemma. 

\begin{thm}[Ultrafilter Lemma] \label{thm:UL}
    Let $X$ be a set and $\mathcal{F}$ a filter on $X$. Then $\mathcal{F}$ is contained in an ultrafilter.
\end{thm}

\begin{rem}
    Note that the Ultrafilter Lemma gives us the existence of non-principal ultrafilters, since one can apply it to the Fr\'echet filter. 
\end{rem}

We make the following definitions. 

\begin{defn}
\label{def:finite intersection property}
    Let $X$ be a set and $S \subseteq \Pset(X)$. Then $S$ is said to have the \emph{finite intersection property} if the intersection of finitely many elements of $S$ is never the empty set. 
\end{defn}

We will also say that a subset $S \subseteq \Pset(X)$ generates a filter if the following set is a filter 
\[
\mathcal{F}_S \coloneqq \left\{ A \subseteq X : \bigcap_{i=1}^k S_i \subseteq A \text{ for some }  S_1, \ldots, S_k \in S   \right\}.
\]
Thus $\mathcal{F}_S$ is the smallest ``upwards closed'' subset of $\Pset(X)$ containing all the finite intersections of elements of $S$. Saying that $\mathcal{F}_S$ is a filter is equivalent to saying that $S$ has the finite intersection property. Hence we obtain the following, which is 
the main tool we will use to construct filters, and thereby ultrafilters. 

\begin{prop}
\label{prop:filter generation}
    Let $X$ be a set and $S \subseteq \Pset(X)$. Then $S$ generates a filter on $X$ if and only if it satisfies the finite intersection property. 
\end{prop}

In this case, by the Ultrafilter Lemma, $S$ will be contained in some ultrafilter on $X$.

\subsection*{Ultralimits}

The notions of a limit via filters and ultrafilters are fairly general but we shall be concerned with the following more restrictive notion. Throughout this one should regard $X$ as an indexing set.

\begin{defn}
\label{def:limitviafilter}
    Let $X$ be a set, $f\colon X \to \R$ a function and $\mathcal{F}$ a filter on $X$. Then for $a \in \R$, we say that the \emph{limit} of $f$ with respect to $\mathcal{F}$ is $a$, written $\lim_{\mathcal{F}} f = a$,  if for all $0 < \epsilon \in \R$, 
    \[
    \{ x \in X : |f(x) - a| < \epsilon \} \in \mathcal{F}.
    \]
    In the case where $\mathcal{F}$ is an ultrafilter, this is called the \emph{ultralimit} of $f$ with respect to $\mathcal{F}$. 
\end{defn}

\begin{rem}
    We are not guaranteed that this limit exists, but when it does it is unique since the topology on $\R$ is Hausdorff. 
\end{rem}

\begin{rem}
We note that Definition~\ref{def:limitviafilter} can also be obtained in the following way; let $Y$ be a topological space and $\mathcal{F}_Y$ a filter on $Y$. Then we say that $\mathcal{F}_Y \to y \in Y$ if every (open) neighbourhood of $y$ is an element of $\mathcal{F}_Y$. 

If $f\colon X \to Y$ is a function and $\mathcal{F}_X$ is a filter on $X$ then the push-forward of $\mathcal{F}_X$ is $f^*{\mathcal{F}_X} \coloneqq \{ A \subseteq Y : f^{-1}(A) \in \mathcal{F}_X\}$. This is a filter on $Y$ (and an ultrafilter when $\mathcal{F}_X$ is an ultrafilter).  

The notion of limit in Definition~\ref{def:limitviafilter} is the same as saying that $f^*{\mathcal{F}_X} \to y$. However, the formulation in Definition~\ref{def:limitviafilter} suits us better since we will keep the set $X$ constant and consider various functions $f$ whose limits are all defined via the same ultrafilter on $X$. In particular, we will be able to utilise an algebra of limits---Lemma~\ref{lem:algebraoflimits}. 
\end{rem}

The following is straightforward to prove: 

\begin{lem}
\label{lem:algebraoflimits}
    Let $X$ be a set and $\mathcal{U} $ an ultrafilter on $X$.
    \begin{enumerate}
        \item If $f\colon X \to \R$ is bounded, then $\lim_{\mathcal{U}} f$ exists and is unique. 
        \item If $f,g\colon X \to \R$ and $a, b \in \R$, then $\lim_{\mathcal{U}}(af+bg) = a \lim_{\mathcal{U}} f + b \lim_{\mathcal{U}} g$, whenever these limits exist. 
        \item If $\mathcal{U}$ is non-principal then $\lim_{\mathcal{U}} f$ is a limit point of $f$, where $y \in \R$ is called a \emph{limit point} of $f$ if $f^{-1}(O)$ is infinite for any open neighbourhood $O$ of $y$. 
        \item Conversely, if $y$ is a limit point of $f$, then there is some non-principal ultrafilter $\mathcal{U}$ such that $\lim_{\mathcal{U}} f = y$. 
        \item If $\mathcal{U}$ is principal and hence generated by some $\{ x_0 \}$, then $\lim_{\mathcal{U}}f = f(x_0)$.
    \end{enumerate}
\end{lem}

\begin{rem}
    More generally, limits via filters always obey the algebra of limits when these exist and make sense. That is, these limits preserve sums and products in the usual way. 
\end{rem}

\subsection*{Finite-by-abelian-by-finite (FAF) groups}

We end these preliminaries by a short discussion of finite-by-abelian-by-finite groups, which we shall be analysing in significant detail. For now, we simply justify the lack of bracketing in the terminology. 

As in the introduction, given classes $\mathcal{A}$ and $\mathcal{B}$ of groups, a group $G$ is said to be $\mathcal{A}$-by-$\mathcal{B}$ if $G$ admits a normal subgroup $N$ so that $N$ is in $\mathcal{A}$ and $G/N$ is in $\mathcal{B}$. We also recall that a group is called virtually $\mathcal{A}$ if it admits a finite index subgroup in $\mathcal{A}$. 

\begin{rem} \label{rem:core}
If $\mathcal{A}$ is a class of groups closed under taking (finite index) subgroups, then a virtually $\mathcal{A}$ group is the same as an $\mathcal{A}$-by-finite group, as the following construction shows.  For a group $G$ and a subgroup $H \leq G$, we define the core of $H$ (in $G$) as $\Core_G(H) \coloneqq \bigcap_{g \in G} H^g$---the intersection of all the conjugates of $H$ in $G$; thus $\Core_G(H)$ is the largest normal subgroup of $G$ contained in $H$.  Crucially, if $H$ is a finite index subgroup of~$G$ then $\Core_G(H)$ is also a finite index subgroup of $G$ which is now normal. Hence a virtually $\mathcal{A}$ group is the same as an $\mathcal{A}$-by-finite group, as claimed. For instance, this will be the case when the class $\mathcal{A}$ consists of abelian groups.
\end{rem}

By a \emph{FAF group} we mean a finite-by-abelian-by-finite group.  The following result shows that this definition does not depend on the chosen bracketing of the phrase ``finite-by-abelian-by-finite''.  Recall that a group is \emph{residually finite} if the intersection of all its (normal) finite index subgroups is trivial. 

\begin{lem}[Equivalence of bracketings for FAF groups: (FA)F = FAF = F(AF)] \label{lem:faf-equivalent}
    Let $G$ be a group. Then $G$ is (finite-by-abelian)-by-finite if and only it is finite-by-(abelian-by-finite).  Moreover, if a FAF group $G$ is either finitely generated or residually finite, then $G$ is abelian-by-finite.
\end{lem}

\begin{proof}
    Suppose first that $G$ is finite-by-(abelian-by-finite).  Then we have short exact sequences of groups
    \[
    1 \to N \to G \xrightarrow{\pi_H} H \to 1 \qquad\text{and}\qquad 1 \to A \to H \xrightarrow{\pi_Q} Q \to 1
    \]
    with $N$, $Q$ finite and $A$ abelian.  Viewing $A$ as a normal subgroup of $H$ and restricting the first sequence to $G_0 \coloneqq \pi_H^{-1}(A)$, we get a short exact sequence $1 \to N \to G_0 \to A \to 1$ and so $G_0$ is finite-by-abelian, while $G/G_0 \cong Q$ is finite.  This shows that $G$ is (finite-by-abelian)-by-finite.

    Conversely, let $G$ be a (finite-by-abelian)-by-finite group.  Then we have a short exact sequence
    \[
    1 \to H \to G \to Q \to 1
    \]
    with $H$ finite-by-abelian and $Q$ finite.  Note that the fact that $H$ is finite-by-abelian means precisely that the derived subgroup $H' = [H,H]$ is finite.  Note that $H'$ is characteristic in $H$, so since $H$ is normal in $G$ it follows that $H'$ is normal in $G$ as well.  Since $H/H'$ is abelian and $G/H$ is finite, it follows that $G/H'$ is abelian-by-finite.  Thus $G$ is finite-by-(abelian-by-finite), as required.

    Suppose now that $G$ is a finitely generated FAF group. Then $G$ has a finite index finite-by-abelian normal subgroup $H$, and since $G$ is finitely generated, so is $H$; let $\{s_1,\ldots,s_k\}$ be a finite generating set of $H$.  Since $H$ is finite-by-abelian it follows that $|H'| < \infty$, and in particular the conjugacy class of each $s_i$, $\{ g^{-1}s_ig : g \in G \} = \{ s_i [s_i,g] : g \in G \}$, is contained in $s_iH'$ and so is finite.  It follows that $C_H(s_i)$ has finite index in $H$ for $1 \leq i \leq k$, and therefore the centre $Z(H) = \bigcap_{i=1}^k C_H(s_i)$ also has finite index in $H$.  Since $Z(H)$ is characteristic in $H$ and $H$ is normal in $G$, it follows that $Z(H)$ is normal in $G$, and since both $[H:Z(H)]$ and $[G:H]$ are finite, it follows that so is $[G:Z(H)]$.  As $Z(H)$ is clearly abelian, it follows that $G$ is abelian-by-finite, as required.

    Finally, suppose that $G$ is a residually finite FAF group.  Then $G$ has a finite normal subgroup $N$ such that $G/N$ is abelian-by-finite.  Since $N$ is finite and $G$ is residually finite, $G$ contains a finite index subgroup $H$ such that $H \cap N = \{1\}$.  It follows that the composite $H \hookrightarrow G \twoheadrightarrow G/N$ is injective, implying that $H$ is isomorphic to a subgroup of $G/N$.  Since $G/N$ is abelian-by-finite, it follows that $H$ has a finite index abelian subgroup $A_0$, which is thus a finite index abelian subgroup of $G$.  In particular, $G$ has a normal finite index abelian subgroup $A = \Core_G(A_0)$, and so is abelian-by-finite, as required.
\end{proof}

\begin{rem} \label{rem:FCAF}
    It can be seen from the proof that a FAF group $G$ contains a finite index normal subgroup $H$ that has a finite derived subgroup $H'$.  The fact that $H'$ is finite then implies that so is $\Aut(H')$, and in particular the kernel of the map $H \to \Aut(H')$ sending $h \in H$ to conjugation by $h$ has finite index in $H$.  But this kernel is precisely $C_H(H')$; moreover, since $H'$ is characteristic in $H$, so is $C_H(H')$.  In particular, $C_H(H')$ is characteristic and of finite index in $H$, and $H$ is normal and of finite index in $G$, implying that $C_H(H')$ is a normal finite index subgroup of $G$.  Moreover, $C_H(H')/Z(H')$ is isomorphic to the subgroup $C_H(H')H'/H'$ of $H/H'$, and in particular is abelian.  It follows, by considering $N_0 = Z(H')$ and $H_0 = C_H(H')$, that any FAF group $G$ has a finite normal subgroup $N_0 \unlhd G$ and a finite index normal subgroup $H_0 \unlhd G$ containing $N_0$ such that $N_0$ is central in $H_0$ and $H_0/N_0$ is abelian.

    That is, any FAF group is virtually a $2$-step nilpotent group whose derived subgroup (which is central) is finite. 
\end{rem}

\section{Construction of coset correct means}
\label{sec:CCMs}

The goal of this section is to prove that every group admits a CCM, defined below.  Our argument here is via ultrafilters, using the Ultrafilter Lemma (Theorem~\ref{thm:UL}), but we also give an argument using the Hahn--Banach Theorem (Theorem~\ref{thm:hahn-banach}) in the Appendix.

\begin{defn} \label{def:ccm}
Let $G$ be a group. A \emph{coset correct mean} (\emph{CCM}) on $G$ is a function $\mu \colon \Pset(G) \to [0,1]$ satisfying the following properties: 
\begin{enumerate}
    \item \label{zeroth-condition} $\mu(G) =1$ (\emph{probability mean});
    \item \label{first-condition} if $A, B \subseteq G $ are disjoint, then $\mu(A \sqcup B) = \mu(A) + \mu(B)$ (\emph{finitely additive});
    \item \label{second-condition} $\mu(gH) = \frac{1}{[G:H]}$ for all subgroups $H \leq G$ and all $g \in G$ (\emph{left invariant on subgroups)}.
\end{enumerate}
As always, $\frac{1}{[G:H]}$ is understood to be zero when $H$ has infinite index.    
\end{defn}

More generally, let $\mathcal{R}$ be a subring of the Boolean ring $\Pset(G)$.  We similarly define a \emph{coset correct mean} (\emph{CCM}) on $\mathcal{R}$ to be a function $\mu\colon \mathcal{R} \to [0,1]$ satisfying the conditions \ref{zeroth-condition}--\ref{second-condition} above whenever $\mu$ is defined on the corresponding subsets of $G$.

\begin{rem} \label{rem:left equals right}
    We also note that a \textit{left} CCM must also be a \textit{right} CCM since 
    \[
    \mu(Hg) = \mu(g (g^{-1} H g)) = \frac{1}{[G: H^g]} = \frac{1}{[G:H]} = \mu(H). 
    \]
That is, left invariance on cosets is equivalent to right invariance on cosets. 
\end{rem}

The strategy in our construction is the following: given a group $G$, we denote by $\Pfin(G)$ the set of finite subsets of $G$, which will serve as an ``indexing set''. We then define functions for each $A \subseteq G$, $f_A\colon \Pfin(G) \to [0,1]$, via
\[
f_A(F) = \frac{|A \cap F|}{|F|}.
\]
We then take an ultrafilter $\mathcal{U}$ on $\Pfin(G)$ and define $\mu(A) = \lim_{\mathcal{U}} f_A$. It is immediate that $\mu$ has total weight one, since $f_G$ is the constant function and it is also clear that $\mu$ is finitely additive since $f_{A \sqcup B} = f_A + f_B$ (using the algebra of ultralimits, Lemma~\ref{lem:algebraoflimits}). 

All the remains to show is that $\mu$ is left invariant on cosets. This is not true for all ultrafilters so the goal is to find some (ultra)-filter where it is true. However, the property that $\mu$ is left invariant on cosets determines certain subsets of $\Pfin(G)$ that would need to belong to the (ultra)-filter and hence our strategy becomes simply showing that these sets generate a filter, as in Proposition~\ref{prop:filter generation}. We actually show this in some strong sense in the following:

\begin{lem} \label{lem:index-consistent-filter}
Given a subgroup $H \leq G$ and a real number $\varepsilon > 0$, we define
\[
\mathcal{F}_\varepsilon(H) \coloneqq \left\{ F \in \Pfin(G) : \left| \frac{|gH \cap F|}{|F|} - \frac{1}{[G:H]} \right| < \varepsilon \text{ for all } g \in G \right\}.
\]
Then the collection $\mathcal{G} = \mathcal{G}_G \coloneqq \{ \mathcal{F}_\varepsilon(H) : H \leq G, \varepsilon > 0 \}$ generates a filter on $\Pfin(G)$, the set of finite subsets of $G$.
\end{lem}

\begin{proof}
Note that, by Proposition~\ref{prop:filter generation}, it is enough to show that $\mathcal{G}$ satisfies the finite intersection property, Definition~\ref{def:finite intersection property}. Thus, let $\varepsilon_1,\ldots,\varepsilon_k > 0$ and $H_1,\ldots,H_k \leq G$. We then need to show that $\bigcap_{i=1}^k \mathcal{F}_{\varepsilon_i}(H_i) \neq \varnothing$.

After possibly permuting the indices, we may assume that, for some $m$ with $0 \leq m \leq k$, we have $[G:H_i] = \infty$ if $i \leq m$ and $[G:H_i] < \infty$ if $i > m$. Then $H \coloneqq \bigcap_{i=m+1}^k H_i$ is a finite index subgroup of $G$. Let $N = r[G:H]$ for some integer $r \geq 1$ large enough so that $\frac{1}{N} < \varepsilon_i$ for $1 \leq i \leq m$. We aim to construct a subset $F \in \mathcal{F}$ of cardinality $N$ that contains at most one element in each left coset of $H_i$ for $i \leq m$, and exactly $r$ elements in each left coset of $H$. This will imply that for each $g \in G$, we have $\frac{|gH_i \cap F|}{|F|} \leq \frac{1}{N} < \varepsilon_i$ for $i \leq m$, and $\frac{|gH_i \cap F|}{|F|} = \frac{r[H_i:H]}{r[G:H]} = \frac{1}{[G:H_i]}$ for $i > m$, and so $F \in \bigcap_{i=1}^k \mathcal{F}_{\varepsilon_i}(H_i)$, as required.

Let $H = g_1H,\ldots,g_nH$ be all the left cosets of $H$, where $n = [G:H]$. We construct the subset $F = \{ f_1,\ldots,f_N \}$, where $f_{ir+j} \in g_{i+1}H$ for $0 \leq i \leq n-1$ and $1 \leq j \leq r$, inductively as follows. Having constructed $\{ f_1,\ldots,f_{s-1} \}$ for some $s \geq 1$, let $f_s \in g_iH$, where $i = \lceil \frac{s}{r} \rceil$, be such that $f_sH_j \neq f_tH_j$ for any $j \leq m$ and $t < s$. Such a choice is possible: indeed, if it wasn't then we would have 
\[
g_iH \subseteq \bigcup_{j=1}^m \bigcup_{t=1}^{s-1} f_tH_j,
\]
i.e.\ $g_iH$ could be covered by a finite union of cosets of its infinite index subgroups of $G$, which is impossible by Theorem~\ref{thm:neumann} since $H$ has finite index in $G$.

Thus the subset $F = \{ f_1,\ldots,f_N \} \in \bigcap_{i=1}^k \mathcal{F}_{\varepsilon_i}(H_i)$ exists, as required.
\end{proof}

\begin{thm}
\label{thm: existence of CCMs}
    The Ultrafilter Lemma implies that every group $G$ admits a CCM.
\end{thm}

\begin{proof}
    Let $\mathcal{U}$ be any ultrafilter containing $\mathcal{G}_G$ from Lemma~\ref{lem:index-consistent-filter}, the existence of which follows from the Ultrafilter Lemma (Theorem~\ref{thm:UL}). As above, we have functions for each $A \subseteq G$, $f_A: \Pfin(G) \to [0,1]$ defined by
    \[
    f_A(F) = \frac{|A \cap F|}{|F|}.
    \]
    We then set, 
    \[
    \mu(A) \coloneqq \lim_{\mathcal{U}} f_A. 
    \]
    From the remarks above, it is clear that this is a finitely additive probability mean. Moreover, by construction, $\mu(gH) = \mu(H)$ for all subgroups $H$ and all $g \in G$. Hence $\mu$ is a CCM.
\end{proof}

\begin{rem}
Note that as long as $G$ is infinite, for any $\varepsilon > 0$ the subset $\mathcal{F}_\varepsilon(\{1\}) \in \mathcal{G}$ does not contain any $F \in \Pfin(G)$ of cardinality $\leq \frac{1}{\varepsilon}$. This implies that the filter constructed above is non-principal. 

Alternatively, if $\mathcal{U}$ is a principal ultrafilter on $\Pfin(G)$ then it is generated by some fixed $F_0 \in \Pfin(G)$; that is, $\mathcal{U} = \langle F_0 \rangle$. Then,  

\[
\lim_{\langle F_0 \rangle} \frac{|F_0 \cap F|}{|F|} = \frac{|F_0 \cap F_0|}{|F_0|} = 1.
\]
But $F_0$ is a disjoint union of cosets of the trivial subgroup and so a CCM will assign it density $\frac{|F_0|}{[G: \{ 1\} ]} = \frac{|F_0|}{|G|}$. Hence $\mathcal{U} = \langle F_0 \rangle$  can only define a CCM on $G$ if $F_0=G$ is finite.  

Thus we are observing that, when $G$ is infinite, our construction produces a non-principal ultrafilter and that \textit{any} ultrafilter used to construct a CCM must also be non-principal. 
\end{rem}

\begin{rem}
Instead of the subsets $\mathcal{F}_\varepsilon(H)$, given any $H \leq G$, $g \in G$ and $\varepsilon > 0$ one may consider the subset $\mathcal{F}_\varepsilon'(gH) \coloneqq \left\{ F \in \Pfin(G) : \left| \frac{|gH \cap F|}{|F|} - \frac{1}{[G:H]} \right| < \varepsilon \right\}$. Since clearly $\mathcal{F}_\varepsilon'(gH) \supseteq \mathcal{F}_\varepsilon(H)$ for all $H$, $g$ and $\varepsilon$, it follows that the collection $\mathcal{G}' \coloneqq \{ \mathcal{F}_\varepsilon'(gH) : H \leq G, g \in G, \varepsilon > 0 \}$ generates a filter (that is a sub-filter of that generated by $\mathcal{G}$).

In a sense, the filter we construct in Lemma~\ref{lem:index-consistent-filter} is the uniform version of what we need, as it deals with all cosets of a given subgroup at the same time rather than one-by-one. 
\end{rem}

While the following results are known, we present elementary proofs that are particularly suited to our context. This is not merely for completeness: the combinatorial nature of our setting, especially the role played by finite sets and the Boolean algebra they generate, makes it natural to revisit these results from a more foundational perspective. The reader may wish to compare with \cite{Hopfensperger}, particularly Theorems 2.6 and 5.9, where related arguments are given in a broader setting.

\begin{prop}
\label{prop:means are given by ultralimits}
    Let $X$ be a set and let $\mu$ be a finitely additive probability mean on $X$. That is, $\mu\colon \Pset(X) \to [0,1]$ is a function such that (i) $\mu(X)=1$ and (ii) $\mu(A \sqcup B) = \mu(A) + \mu(B)$ for all disjoint subsets $A, B $ of $X$. Moreover, suppose that $\mu(F) =0$ for any finite subset $F$ of $X$. 

    Then there exists an ultrafilter $\mathcal{U}$ on $\Pfin(X)$ such that for all $ A \subseteq X$,
    \[
    \mu(A) = \lim_{\mathcal{U}} \frac{|A \cap F|}{|F|}.
    \]
\end{prop}

\begin{proof}
    For any $\epsilon > 0$ and any $A \subseteq X$, consider the sets 
    \[
    \mathcal{F}_\varepsilon(A) = \left\{ F \in  \Pfin(X) :  \left| \frac{|F \cap A|}{|F|} - \mu(A) \right| < \varepsilon \right\} \subseteq \Pfin(X).
    \]
    Note that if these sets satisfy the finite intersection property (Definition~\ref{def:finite intersection property}) then they will be contained in some ultrafilter $\mathcal{U}$ on $\Pfin(X)$. It then follows, by construction, that 
    \[
    \mu(A) = \lim_{\mathcal{U}} \frac{|A \cap F|}{|F|}.
    \]
    Hence it is enough to show that these sets satisfy the finite intersection property. 

    To that end, let $A_1,\ldots,A_k \subseteq X$ and $\varepsilon_1,\ldots,\varepsilon_k > 0$. We then need to show that we have $\bigcap_{i=1}^k \mathcal{F}_{\varepsilon_i}(A_i) \neq \varnothing$.

    Given any function $f\colon \{1,\ldots,k\} \to \{0,1\}$, consider $B_f = \bigcap_{i=1}^k A_i^{(f(i))} \subseteq X$, where $A_i^{(0)} = X - A_i$ and $A_i^{(1)} = A_i$. Then the sets $B_f$ are pairwise disjoint and cover $X$. In other language, we are taking the Boolean subalgebra of $\Pset(X)$ generated by the $A_i$ and considering the atoms, which are exactly the $B_j$.

    Since $A_i$ is the union of all those $B_f$ for which $f(i) = 1$, it is clear by finite additivity of $\mu$ and the triangle inequality that we have
    \[
    \mathcal{F}_{\varepsilon_i}(A_i) \supseteq \bigcap_{\substack{f\colon \{1,\ldots,k\} \to \{0,1\} \\ f(i)=1}} \mathcal{F}_{\varepsilon_i/2^{k-1}}(B_f).
    \]
    In particular, if the intersection $\bigcap_{f\colon \{1,\ldots,k\} \to \{0,1\}} \mathcal{F}_\varepsilon(B_f)$ is non-empty, where $\varepsilon = \frac{\min\{\varepsilon_1,\ldots,\varepsilon_k\}}{2^{k-1}}$, then so is $\bigcap_{i=1}^k \mathcal{F}_{\varepsilon_i}(A_i)$.  We can therefore assume, without loss of generality, that the $A_i$ are pairwise disjoint and cover $X$, and that all the $\varepsilon_i$ are equal to a fixed $\varepsilon > 0$.

    Now suppose that $A_i$ are ordered so that we have $\mu(A_i) > 0$ for $i \leq m$ and $\mu(A_i) = 0$ for $i > m$. It then follows that $A_i$ is infinite for all $i \leq m$. Pick an integer $N > \frac{1}{\varepsilon}$, and for $1 \leq i \leq m$ let $F_i' \subset A_i$ be any subset of cardinality $\lfloor \mu(A_i)N \rfloor$. Note that we have $\sum_{i=1}^m \mu(A_i) = \sum_{i=1}^k \mu(A_i) = \mu(X) = 1$ since the $A_i$ are pairwise disjoint and cover $X$ and since $\mu$ is finitely additive.  It follows that
    \[
    N - m = \sum_{i=1}^m (\mu(A_i)N-1) \leq \sum_{i=1}^m \lfloor \mu(A_i)N \rfloor = \sum_{i=1}^m |F_i'| \leq \sum_{i=1}^m \mu(A_i)N = N.
    \]
    This implies that we have subsets $F_i \subset A_i$, each obtained by adding at most one element to $F_i'$, such that we have $\sum_{i=1}^m |F_i| = N$. Note that $\lfloor \mu(A_i)N \rfloor \leq |F_i| \leq \lfloor \mu(A_i)N \rfloor + 1$, and consequently $\big| |F_i|-\mu(A_i)N \big| \leq 1$.

    We now claim that $F = \bigcup_{i=1}^m F_i$ belongs to $\mathcal{F}_\varepsilon(A_i)$ for each $i$. Indeed, for $i > m$ we have $F \cap A_i = \varnothing$ and $\mu(A_i) = 0$, implying that $\left| \frac{|F \cap A_i|}{|F|} - \mu(A_i) \right| = 0 < \varepsilon$. On the other hand, for $i \leq m$ we have $F \cap A_i = F_i$ and therefore
    \[
    \left| \frac{|F \cap A_i|}{|F|} - \mu(A_i) \right| = \left| \frac{|F_i|}{N} - \mu(A_i) \right| = \frac{\big| |F_i| - \mu(A_i)N \big|}{N} \leq \frac{1}{N} < \varepsilon,
    \]
    as claimed. Thus we indeed have $\bigcap_{i=1}^k \mathcal{F}_\varepsilon(A_i) \neq \varnothing$.

    As we observed at the start, this is sufficient to prove the result. 
\end{proof}

\begin{rem}
    It is easy to verify that if $X=G$ is a group and the ultrafilter constructed above satisfies the following F{\o}lner condition, namely: $\lim_{\mathcal{U}} \frac{|F \oplus gF|}{|F|} =0$ for all $g \in G$, then the mean constructed will be a (left) invariant mean and $G$ will be amenable. However, notice that in the construction we can always insist that our finite sets satisfy $ F \cap gF = \varnothing$ and so $\lim_{\mathcal{U}} \frac{|F \oplus gF|}{|F|} = 2$, even if the mean itself is (left) invariant. We now do this explicitly.
\end{rem}

\begin{cor}
\label{cor:disjoint subsets}
Let $G$ be an infinite group and $\mu$ a finitely additive probability mean on $G$ with $\mu(F)=0$ for all finite $F\subset G$.  
Then, using the construction of the finite approximating sets in the proof of Proposition~\ref{prop:means are given by ultralimits}, we may choose each finite set $F$ so that, for any finite $S\subset G$ and any finite family of subsets $\mathcal{A}\subset \Pset(G)$,
\[
F \cap sF = \varnothing \quad \text{for all } 1\neq s\in S,
\]
while simultaneously 
\[
\left| \frac{|F \cap A|}{|F|} - \mu(A) \right| < \varepsilon
\quad \text{for all } A \in \mathcal{A}.
\]
In particular, there is an ultrafilter $\mathcal{U}$ on $\Pfin(G)$ such that
\[
\mu(A) = \lim_{\mathcal{U}} \frac{|F \cap A|}{|F|} \quad \text{and} \quad \lim_{\mathcal{U}} \frac{|F \oplus gF|}{|F|} = 2 \ \forall g \in G.
\]
\end{cor}

\begin{proof}
The corollary follows directly from the construction in the proof of Proposition~\ref{prop:means are given by ultralimits}, with a minor refinement to ensure the disjointness condition for finite subsets, $S$ of $G$. The existence of the ultrafilter then follows from Proposition~\ref{prop:filter generation} and Theorem~\ref{thm:UL}. 

As in the proof of the Proposition, for a finite family of sets $\mathcal{A} = \{A_1,\dots,A_k\}$, we consider the Boolean algebra they generate and its atoms $B_1,\dots,B_m$.  
For each atom $B_j$, we choose a finite number $p_j$ of elements to include in $F$, where $p_j>0$ only if $B_j$ is infinite, and $p_j$ is determined by the required approximation of $\mu(A_i)$.

The key observation is that each $B_j$ from which we choose is infinite, so we are free to avoid finitely many “forbidden” elements at each stage. In particular, as we select the $p_j$ elements from $B_j$, we can always ensure that $F \cap sF = \varnothing$ for all $1\neq s \in S$: explicitly, we may pick elements one-by-one from the complement of $\bigcup_{s \in S}(F_0 \cup sF_0 \cup s^{-1}F_0)$, where $F_0 \subseteq F$ is the subset already constructed at a certain stage.  Since the number of forbidden elements at each stage is finite, and each ``operative'' $B_j$ is infinite, this process can be carried out without affecting the required counts $p_j$. Thus the resulting finite set $F$ simultaneously approximates $\mu$ on all $A_i$ and satisfies the disjointness conditions $F \cap sF = \varnothing$. 
\end{proof}

We then show that for an amenable group, any invariant mean may be obtained via an ultralimit that does satisfy the F{\o}lner condition. To do that, we use the classical F{\o}lner condtion. 

\begin{prop}[\cite{Folner1955}]
\label{prop:folner condition}
    Let \( G \) be a group. Then \( G \) is amenable if and only if for every finite subset \( K \subseteq G \) and every \( \varepsilon > 0 \), there exists a non-empty finite subset \( S \subseteq G \) such that
\[
\frac{|S \oplus gS|}{|S|} < \varepsilon \quad \text{for all } g \in K.
\]
\end{prop}

\begin{prop}
    \label{prop:invariant means and Folner}
Let $G$ be an infinite amenable discrete group and let $\mu\colon \Pset(G) \to [0,1]$ be a left invariant mean on $G$. Then there exists an ultrafilter $\mathcal{U}$ on $\Pfin(G)$ which both realises $\mu$ and satisfies the F{\o}lner criterion as follows, 
\[
\mu(A) = \lim_{\mathcal{U}} \frac{|A \cap F|}{|F|} \quad \text{ and } \quad \lim_{\mathcal{U}} \frac{|F \oplus gF|}{|F|} = 0 \ \forall g \in G.
\]
\end{prop}
\begin{proof}
    Note that as $G$ is an infinite amenable group, $\mu$ returns a density of $0$ for any finite subset of $G$ and hence we may use Corollary~\ref{cor:disjoint subsets}, which we shall do below.  

    To construct the ultrafilter we invoke Proposition~\ref{prop:filter generation} as before. It is clearly sufficient to verify the finite intersection property for the following family of subsets,
    \begin{align*}
    \mathcal{F}_{\varepsilon}(\mathcal{A}) &\coloneqq \left\{ F \in \Pfin(G) : \left|\tfrac{|F \cap A_i|}{|F|} - \mu(A_i)\right| < \varepsilon \: \forall A_i \in \mathcal{A} \right\}, \ \mathcal{A} = \{A_1,\dots,A_r\} \subseteq \Pset(G), \ \varepsilon > 0, \\
    \mathcal{E}_{\varepsilon}(K) &\coloneqq \left\{ F \in \Pfin(G) : \tfrac{|F \oplus gF|}{|F|} < \varepsilon \: \forall g \in K \right\}, \  K \in \Pfin(G), \ \varepsilon > 0.
    \end{align*}

    Notice that Corollary~\ref{cor:disjoint subsets} already gives us that $\mathcal{F}_{\varepsilon}(\mathcal{A})$ is non-empty for any $\mathcal{A} = \{ A_1, \ldots, A_r\} \subset \Pset(G) $ and any $\varepsilon > 0$. One easily observes that to verify the finite intersection property for this family, it is enough to show that 
    \[
    \mathcal{F}_{\varepsilon} (\mathcal{A})  \cap \mathcal{E}_{\varepsilon}(K)  \neq \varnothing \quad \text{for any } \mathcal{A}, K \text{ and } \varepsilon>0. 
    \]
    So let's assume that $\mathcal{A}$, $K$ and $\varepsilon$ are given and proceed to show that the intersection above is non-empty. 

    We first invoke Proposition~\ref{prop:folner condition} to find a finite subset $S$ of $G$ such that $|S \oplus gS| < \varepsilon |S| $ for all $g \in S$. Now we invoke Corollary~\ref{cor:disjoint subsets} to find an $F_0 \in \mathcal{F}_{\varepsilon} (K \mathcal{A})$ such that $F_0 \cap gF_0 = \varnothing$ for all $1 \neq g \in S^{-1}S$, where $K \mathcal{A} \coloneqq \{ kA_i : k \in K, A_i \in \mathcal{A} \} $. We claim that $F \coloneqq SF_0 = \bigcup_{g \in S} gF_0$ is in  $\mathcal{F}_{\varepsilon} (\mathcal{A})  \cap \mathcal{E}_{\varepsilon}(K)$. 

    To see this, note that  $F_0 \cap gF_0 = \varnothing$ for all $1 \neq g \in S^{-1}S$ implies that the sets $gF_0$, for $g \in S$, are pairwise disjoint. Hence $|F| = |S||F_0|$. In particular, $|F \oplus gF| \leq |S \oplus gS| |F_0| < \varepsilon |S||F_0| = \varepsilon |F|$ for all $g \in K$. Therefore, $F \in \mathcal{E}_{\varepsilon}(K)$. To see that $F \in \mathcal{F}_{\varepsilon}(\mathcal{A})$, we calculate
    \begin{align*}
    \big|\,|F \cap A_i| - \mu(A_i)|F|\,\big|
    &\leq \sum_{g \in S} \big|\,|gF_0 \cap A_i| - \mu(A_i)|F_0|\,\big| && \text{as } |F| = |S||F_0| \\
    &= \sum_{g \in S} \big|\,|F_0 \cap g^{-1}A_i| - \mu(g^{-1}A_i)|F_0|\,\big| && \text{as } \mu(A_i) = \mu(g^{-1}A_i) \\
    &< \sum_{g \in S} \varepsilon|F_0| && \text{since } F_0 \in \mathcal{F}_{\varepsilon}(K\mathcal{A}) \\
    &= \varepsilon|F|.
    \end{align*}
    Hence $F$ is also in $\mathcal{F}_{\varepsilon} (\mathcal{A}) $ and we are done. 
\end{proof}

\begin{rem}
The construction above uses a two-scale idea: first, we choose a ``core'' set $F_0$ that approximates the mean on all relevant translates of the given subsets and satisfies strong disjointness conditions. Then we amplify $F_0$ by multiplying with a large F{\o}lner set $S$  to form $F = S F_0$. This ensures that the relative boundary $\frac{|F \oplus gF|}{|F|}$ becomes arbitrarily small while preserving the density approximations, thanks to the invariance of $\mu$. In essence, the invariance allows us to control averages over translates, and the use of $S$ maintains the F{\o}lner property.
\end{rem}

\section{Product means and a degree of commutativity}
\label{sec:product}

As an application of the existence of a CCM, we use it to measure the ``probability'' that two elements of a group commute and eventually derive structural results about groups from this quantity. Our point of view is to understand the probability of $\Comm(G)$ in $G^2$ and so we extend our CCM from $G$ to $G^2$ via a product mean. 

\begin{defn} \label{defn: product mean}
Suppose that $\mu$ is a CCM on a group $G$.  We may then define a mean $\mu^{(2)}$ on $G^2$, which we refer to as a \emph{product mean}, by integrating: given $S \subseteq G^2$ we set
\[
\mu^{(2)}(S) \coloneqq \int \mu(S_g) \,\mathrm{d}\mu(g),
\]
where $S_g = \{ h \in G : (g,h) \in S \}$; see \cite[Chapter~III]{DunfordSchwartz1988}.
\end{defn}

\begin{rem} \label{rem:product-asymmetric}
    There is certain asymmetry in this definition.  Indeed, suppose that $\mu$ is a CCM on a group $G$.  We could have also defined a mean $\widehat\mu^{(2)}$ on $G^2$ by setting
    \[
    \widehat\mu^{(2)}(S) \coloneqq \int \mu(\widehat{S}_g) \,\mathrm{d}\mu(g),
    \]
    where $\widehat{S}_g = \{ h \in G : (h,g) \in S \}$.  
    
    These two definitions agree when $G$ is finite (in which case $\mu^{(2)}(S) = \frac{|S|}{|G|^2} = \widehat\mu^{(2)}(S)$ for any $S \subseteq G^2$), but not when $G$ is infinite. For instance, if $G = \{g_0,g_1,\ldots \}$ is countable and infinite, then we have $\mu^{(2)}(S) = 1$ but $\widehat\mu^{(2)}(S) = 0$, where $S = \{ (g_i,g_j) \in G^2 : i < j \}$.  A slightly more involved argument shows the same result when $G = \{ g_i : i < \kappa \}$ is any infinite group, where $\kappa$ is the initial ordinal of the cardinal $|G|$.
\end{rem}

We note that if $\mu$ is a CCM, then so is $\mu^{(2)}$.

\begin{lem}
    Let $\mu$ be a CCM on a group $G$.  Then $\mu^{(2)}$ is a CCM on $G^2$.
\end{lem}

\begin{proof}
    It is easy to see that $\mu^{(2)}$ is a finitely additive probability mean.  It remains to show that it is left invariant on cosets.  In particular, given $H \leq G^2$ and $(x,y) \in G^2$, we aim to show that $\mu((x,y)H) = \frac{1}{[G^2:H]}$.

    First we let $H_2 = \{ h \in G : (1,h) \in H \} $ and $S=(x,y) H$. Note that $H_2$ is a subgroup of $G$, and that each $S_g$ (for $g \in G$) is either empty, or a single coset of $H_2$ in $G$. If $H_2$ has infinite index, therefore, we conclude that $\mu^{(2)} ( (x,y) H ) =0$. 

    Otherwise, we have that $H_2$ is a finite index subgroup of $G$. In this case, set $\widehat{H}_1 = \pi_1(H)$, where $\pi_1\colon G^2 \to G$ is the projection to the first coordinate.  Note that $\widehat{H}_1 \times G$ is a finite union of left cosets of $H$ (we can take these cosets to be $(1,g)H$, where $g$ ranges over a left transversal of $H_2$ in $G$).  Therefore, if $H$ has infinite index in $G^2$ then $\widehat{H}_1$ has infinite index in $G$, and consequently
    \[
    \mu^{(2)}((x,y)H) \leq \mu^{(2)}(x\widehat{H}_1 \times G) = \mu(x\widehat{H}_1) \cdot \mu(G) = 0 \cdot 1 = 0,
    \]
    and hence $\mu^{(2)}((x,y)H) = 0$ whenever $H$ has infinite index in $G$.

    In the case where $H$ has finite index in $G$, both $H_1 = \{ g \in G \mid (g,1) \in H \}$ and $H_2$ have finite index in $G$, and we have $H_1 \times H_2 \leq H$. Then every coset of $H_1 \times H_2$ in $G$ has the same mean, $\mu(H_1)\mu(H_2) = \frac{1}{[G:H_1][G:H_2]}$, since $\mu$ is left invariant on cosets. In particular, since each coset of $H$ will be a union of $[H:H_1 \times H_2] = \frac{[G^2:H_1 \times H_2]}{[G^2:H]} = \frac{[G:H_1][G:H_2]}{[G^2:H]}$ cosets of $H_1 \times H_2$, we have
    \[
    \mu^{(2)}((x,y)H) = \frac{[G:H_1][G:H_2]}{[G^2:H]} \mu^{(2)}(H_1 \times H_2) = \frac{1}{[G^2:H]},
    \]
    as required.
\end{proof}

\begin{rem}
    Note that not every CCM on $G^2$ arises as a product mean---for instance, one could take the average of the CCMs $\mu^{(2)}$ and $\widehat\mu^{(2)}$, as in Definition~\ref{defn: product mean} and Remark~\ref{rem:product-asymmetric}. We return to this fact and its relationship to the degree of commutativity in Example~\ref{ex:heisenberg}. 
\end{rem}

Nevertheless, armed with the product mean, we can define a degree of commutativity with respect to it.

\begin{defn}
\label{defn: dc using CCMs}
    Let $\mu$ be a CCM on a group $G$.  We define the \emph{degree of commutativity} of $G$ with respect to $\mu$ as
    \[
    \dc[\mu](G) \coloneqq \mu^{(2)}(\Comm(G)),
    \]
    where $\Comm(G) = \{ (g,h) \in G^2 : [g,h] = 1 \}$.
\end{defn}

We will subsequently prove that this number is the same for \emph{any} CCM $\mu$ in Proposition~\ref{prop:CCM-dc-independent}. 

In the case where $G$ is finite, there is clearly a unique CCM on $G$ and it is already (left) invariant: namely, the uniform measure. Hence the following definition is consistent with the previous one and has been widely studied. 

\begin{defn}
\label{defn: dc for finite groups}
    Let $G$ be a finite group. We define the \emph{degree of commutativity} of $G$ as
    \[
    \dc(G) \coloneqq \frac{\left|\Comm(G)\right|}{|G|^2}. 
    \]
\end{defn}

\section{Groups with positive degree of commutativity are FAF}
\label{sec:dc>0=>FAF}

\subsection*{Finite quotients and residual finiteness}

Since CCMs are, by definition, well behaved for subgroups it is natural to ask about the relationship between $\dc[\mu](G)$ and $\dc(G/N)$, whenever $N$ is a finite index normal subgroup. This leads us naturally to a degree of commutativity which is seen by the finite quotients of a group. This will actually be important for showing that a positive degree of commutativity with respect to a CCM implies the group is FAF, but we need some preliminary results relating this to finite quotients first. 

\begin{defn} \label{defn:dc for RF}
    For a group $G$, we define
    \[
    \dcRF(G) = \inf \{ \dc(G/N) : N \unlhd G, [G:N] < \infty \}.
    \]
    We call this the \emph{residually finite degree of commutativity} for $G$. 
\end{defn}

There is an easy relationship between the residually finite degree of commutativity and that arising from a CCM. 

\begin{lem}
\label{lem: dcCCM less that dcRF}
    Let $G$ be a group and $\mu$ a CCM on $G$. Then $\dc[\mu](G) \leq \dcRF(G)$. 
 \end{lem}
\begin{proof}
    Let $N \unlhd G$ be a normal subgroup of finite index, and let $\pi_N\colon G^2 \to (G/N)^2$ be the natural projection. Then any pair of commuting elements also commute modulo $N$, implying that $\Comm(G) \subseteq \pi_N^{-1}(\Comm(G/N))$ and therefore
    \[
    \mu^{(2)}(\Comm(G)) \leq \mu^{(2)}(\pi_N^{-1}(\Comm(G/N))).
    \]
    The right hand side equals $\dc(G/N)$ because the set $\pi_N^{-1}(\Comm(G/N))$ is a union of exactly $\left|\Comm(G/N)\right|$ cosets of $N^2$ in $G^2$, each of which has measure $1/[G:N]^2 = 1/|G/N|^2$ under $\mu^{(2)}$. Hence,
    \[
    \mu^{(2)}\left(\pi_N^{-1}(\Comm(G/N))\right) = \frac{\left|\Comm(G/N)\right|}{|G/N|^2} = \dc(G/N),
    \]
    and taking the infimum over all such $N$ gives $\dc[\mu](G) \leq \dcRF(G)$.
\end{proof}

Note that $\dcRF(G)$ only really gives information if $G$ admits sufficiently many finite quotients. Residually finite groups are defined as precisely the class of groups where the finite quotients can distinguish elements. More concretely, $G$ is called \emph{residually finite} if $\bigcap_{N \sgpnfin G} N = \{ 1\} $, where we use the symbol $N \sgpnfin G$ to denote that $N$ is a normal subgroup of finite index in $G$. (Via the core construction---see Remark~\ref{rem:core}---one can drop the normality assumption on the finite index subgroups.)

It is then fairly straightforward to see the following. We note that this result fits the pattern of saying that the degree of commutativity is positive exactly when the group is FAF, since residually finite FAF groups are virtually abelian by Lemma~\ref{lem:faf-equivalent}.  

\begin{lem} \label{lem:dcRF0}
    Let $G$ be a residually finite group.  Then $\dcRF(G) > 0$ if and only if $G$ is virtually abelian.
\end{lem}

\begin{proof}
    If $G$ is virtually abelian, then it has an abelian subgroup $H \leq G$ of index $k < \infty$.  Given any finite index normal subgroup $N \unlhd G$, the subgroup $HN/N$ of $G/N$ is an abelian subgroup of index $\leq k$, implying that $\dc(G/N) \geq 1/k^2$.  Thus $\dcRF(G) \geq 1/k^2 > 0$.

    Conversely, suppose that $G$ is not virtually abelian.  We will inductively construct a sequence $G=N_0,N_1,\ldots$ of finite index normal subgroups of $G$ such that $\dc(G/N_i) \leq (5/8)^i$, from which the result will follow.

    Indeed, suppose that for some $i \geq 0$ the group $N_i$ has been constructed.  Since $G$ is not virtually abelian, it follows that $N_i$ is not abelian and so contains elements $g_i,h_i \in N_i$ such that $[g_i,h_i] \neq 1$.  Since $G$ is residually finite, there exists a finite index normal subgroup $M_{i+1} \unlhd G$ such that $[g_i,h_i] \notin M_{i+1}$.  We set $N_{i+1} = N_i \cap M_{i+1}$.

    Then $N_{i+1}$ is a normal finite index subgroup of $G$ since so are $N_i$ and $M_{i+1}$, and $N_{i+1} \unlhd N_i$.  Moreover, by construction $N_i/N_{i+1}$ is non-abelian, and therefore $\dc(N_i/N_{i+1}) \leq 5/8$ by \cite[\S1]{Gustafson1973}.  It then follows by the result of \cite{Gallagher1970} and the inductive hypothesis that
    \[
    \dc(G/N_{i+1}) \leq \dc(G/N_i) \cdot \dc(N_i/N_{i+1}) \leq (5/8)^i \cdot (5/8) = (5/8)^{i+1}, 
    \]
    as required.
\end{proof}

\begin{rem} \label{rem:G/resG}
    For a group $G$ the residual of $G$, $\res(G)$, is defined to be the intersection of all the normal finite index subgroups of $G$. The residual of a group $G$ defines its largest residually finite quotient. The Lemma above then clearly shows that even if $G$ is not residually finite, $\dcRF(G) > 0$ if and only if $G/\res(G)$ is virtually abelian.
\end{rem}

\subsection*{Centralisers and degree of commutativity}

We are now ready to show that having $\dc[\mu](G) > 0$ implies that $G$ is FAF. We will do this by looking at centralisers. 

Throughout the remainder of the paper, we will refer to subsets of $G$ consisting of elements whose centralisers have index $\leq n$, using the following notation.

\begin{defn}
\label{defn:x_n}
    Let $G$ be a group and let $n \in \N$. We denote
    \[
    \G{n} = \{ g \in G : [G:C_G(g)] \leq n \},
    \]
    and we set $\Gfin = \bigcup_{i=1}^\infty \G{n}$.
\end{defn}

Note that for any $g \in \G{n}$ and $h \in \G{m}$, we have
\[
[G:C_G(gh)] \leq [G:C_G(g) \cap C_G(h)] \leq [G:C_G(g)] \cdot [G:C_G(h)] \leq mn,
\]
implying that $\G{n}\G{m} \subseteq \G{nm}$.  In particular, $\Gfin$ is a subgroup of $G$ (and it is normal, since each $\G{n}$ is invariant under conjugation).

Next we observe that having positive degree of commutativity implies that ``many'' elements have centraliser of small index. 

\begin{lem} \label{lem:bigcentraliser}
    Let $\mu$ be a CCM on a group $G$, and suppose that $\dc[\mu](G) = \alpha > 0$. Then $\mu(\G{n}) >0$ for any $n > \alpha^{-1}$.
\end{lem}

\begin{proof}
     We have
     \begin{align*}
         \dc[\mu](G) &= \mu^{(2)}(\{ (x,y) \in G^2 : xy=yx \text{ and } x \in \G{n} \}) \\ &+ \mu^{(2)}(\{ (x,y) \in G^2 : xy=yx \text{ and } x \not\in \G{n} \})
         \\ &\leq \int \mathbbm{1}_{\G{n}} d\mu + \frac{1}{n} \int \mathbbm{1}_G d\mu
         \\ &= \mu(\G{n}) + \frac{1}{n}.
     \end{align*}
    The third line is obtained by estimating $\mu(C_G(x)) \leq 1$ for all $x \in \G{n}$ and $\mu(C_G(x)) < \frac{1}{n}$ for all $x \not\in \G{n}$, as well as $\mathbbm{1}_{\G{n}^c}  \leq  \mathbbm{1}_G$.

    Hence $\mu(\G{n}) > 0$. 
\end{proof}

We now prove that having a positive degree of commutativity implies that a group is FAF. We do this in two stages; we first prove it for $2$-step nilpotent groups (i.e.\ nilpotent groups of class at most~$2$) and then, for the general case, reduce the problem to precisely the class of $2$-step nilpotent groups. 

\begin{rem}
 We recall that a group is nilpotent of class $c$ if its lower central series terminates at the identity in $c$ steps. In particular, $2$-step nilpotent groups have central derived subgroups.
\end{rem}

\begin{lem} \label{lem:dc-finite-by-abelian}
    Let $G$ be a $2$-step nilpotent group equipped with a CCM $\mu$. Suppose that for some $m \in \N$, there exists subset $S \subseteq \G{m}$ such that $\mu(S) > 0$. Then $G$ is FAF.
\end{lem}

\begin{proof}
    Let $k = \lceil \mu(S)^{-1}\rceil$. We prove the result by induction on $(m,k)$, in the lexicographic order.

    For the base case, suppose that $m = 1$. Then we have $S \subseteq Z(G)$, and thus $\langle S \rangle \leq Z(G)$. On the other hand, we have $\mu(\langle S \rangle) \geq \mu(S) > 0$, implying that $\langle S \rangle$ has finite index in $G$. Thus $G$ is virtually abelian and hence FAF, as required.

    Suppose now that $m \geq 2$, and pick any non-central element $s_0 \in S$. The argument splits into two cases, depending on the relationship between $\mu(S \cap C_G(s_0))$ and $\mu(S)$.

    Suppose first that $\mu(S \cap C_G(s_0)) = \mu(S)$. Let $G_1 = C_G(s_0)$ and $S_1 = S \cap G_1$, and define a mean $\mu_1$ on $G_1$ by setting $\mu_1(A) = [G:G_1] \mu(A)$. It is easy to check that $\mu_1$ is a CCM on $G_1$; moreover, for any $s \in S_1$ we have $[G_1:C_{G_1}(s)] = [C_G(s)G_1 : C_G(s)] \leq m$ (note that $C_G(s)$ is normal in $G$ since $G$ is $2$-step nilpotent). Since $s_0$ is non-central, we have $\mu(S_1) \leq \mu(C_G(s_0)) \leq \frac{1}{2}$ and thus
    \[
    \mu_1(S_1)^{-1} = \frac{1}{[G:G_1]\mu(S_1)} \leq \frac{1}{2} \mu(S_1)^{-1} \leq \mu(S_1)^{-1}-1 = \mu(S)^{-1}-1,
    \]
    and so $\lceil \mu_1(S_1)^{-1} \rceil < \lceil \mu(S)^{-1} \rceil$. It then follows by the inductive hypothesis that $G_1$ is FAF.  But $[G:G_1] \leq m$, implying that $G$ is virtually FAF and so FAF as well, as required.

    Suppose now that $\mu(S \cap C_G(s_0)) < \mu(S)$. Since $C_G(s_0)$ has finite index in $G$ and since $\mu$ is finitely additive, it follows that $\mu(S \cap gC_G(s_0)) > 0$ for some $g \in G - C_G(s_0)$. Let $z = [g,s_0]$, and note that, since $G$ is $2$-step nilpotent, the subgroup $\langle z \rangle \leq G$ is central and finite, as the fact that $m_0 \coloneqq [G:C_G(s_0)] \leq m < \infty$ implies that $z^{m_0} = [g^{m_0},s_0] = 1$. Let $G_2 = G/\langle z \rangle$ and $S_2 = (S\langle z \rangle \cap C_G(s_0)) / \langle z \rangle \subseteq G_2$, and define a mean $\mu_2$ on $G_2$ by setting $\mu_2(A\langle z \rangle/\langle z \rangle) = \mu(A\langle z \rangle)$. It is easy to check that $\mu_2$ is a CCM on $G_2$.
    
    We claim that $[G_2:C_{G_2}(s)] \leq \frac{m}{2}$ for all $s \in S_2$. Indeed, given any $g \in G$, we have $|[g,G]| = |g^{-1} \cdot g^G| = |g^G| = [G:C_G(g)]$; similarly, $|[g\langle z \rangle,G_2]| = [G_2:C_{G_2}(g\langle z \rangle)]$. Now for any element $s  = s'\langle z \rangle \in S_2$ (where $s' \in S \cap gC_G(s_0)$), we have $1 \neq z = [g,s_0] = [s',s_0] \in [s',G]$, implying that
    \[
    [G_2:C_{G_2}(s)] = |[s,G_2]| = \left| \frac{[s',G]\langle z \rangle}{\langle z \rangle} \right| = \frac{|[s',G]|}{|\langle z \rangle|} = \frac{[G:C_G(s')]}{|\langle z \rangle|} \leq \frac{m}{|\langle z \rangle|} \leq \frac{m}{2},
    \]
    as claimed.

    Now it follows from the claim that $S_2 \subseteq G_2$ satisfies $\mu_2(S_2) > 0$ and $[G_2:C_{G_2}(s)] \leq \frac{m}{2} \leq m-1$ for all $s \in S_2$. Thus, by the inductive hypothesis, $G_2 = G/\langle z \rangle$ is FAF.  Since $\langle z \rangle$ is finite, it follows that $G$ is FAF as well, as required.    
\end{proof}

We now reduce the general case to the nilpotent case in the following Proposition.

\begin{prop} \label{prop:dc-faf}
    Let $G$ be a group equipped with a CCM $\mu$ such that $\dc[\mu](G) > 0$.  Then $G$ is FAF.
\end{prop}

\begin{proof}
    Choose some $m \in \N$ such that $\frac{1}{\dc[\mu](G)} < m$ and let $S = \G{m}$.  By Lemma~\ref{lem:bigcentraliser}, $\mu(S) > 0$. We have $\mu(\langle S \rangle) \geq \mu(S) > 0$, implying that the subgroup $G_0 \coloneqq \langle S \rangle$ has finite index in $G$.  Note also that $S$ is invariant under conjugacy and so $G_0 \sgpnfin G$.
    
    Now the centre $Z(G_0)$ of $G_0$ is equal to $\bigcap_{s \in S} C_{G_0}(s)$, and in particular it is the intersection of finite index subgroups of $G_0$. Thus $G_0/Z(G_0)$ is residually finite, and hence so is $G_1 \coloneqq G/Z(G_0)$.  
    
    By Lemma~\ref{lem: dcCCM less that dcRF}, $\dcRF(G) \geq \dc[\mu](G)$ and it is clear that $\dcRF(G_1) \geq \dcRF(G)$ and thus $\dcRF(G_1) \geq \dc[\mu](G) > 0$. It follows by Lemma~\ref{lem:dcRF0} that $G_1$ is virtually abelian, and hence $G$ is virtually $2$-step nilpotent.

    Let $G_2 \leq G$ be a finite index $2$-step nilpotent subgroup.  Since $\mu(S) > 0$, it then follows that $\mu(S_0) > 0$, where $S_0 = S \cap gG_2$, for some $g \in G$.  Let $s_0 \in S_0$.  Since $S^{-1}S = \G{m}^2 \subseteq \G{m^2}$, we then get $S_2 \coloneqq s_0^{-1}S_0 \subseteq G_2 \cap \G{m^2} \subseteq \G[G_2]{m^2}$ and moreover $\mu_2(S_2) > 0$, where $\mu_2$ is a mean on $G_2$ defined by $\mu_2(A) = [G:G_2]\mu(s_0A)$. It is easy to see that $\mu_2$ is a CCM on $G_2$ and thus, by Lemma~\ref{lem:dc-finite-by-abelian}, $G_2$ is FAF and hence so is $G$, as required.
\end{proof}

\section{Finite-by-abelian-by-finite (FAF) groups}
\label{sec:FAF=>dc-well-defined}

Since we have shown that positivity of the degree of commutativity corresponds to a group being FAF, the invariance of the actual quantity reduces to the study of it in FAF groups. This section is therefore devoted to showing that in a FAF group, the degree of commutativity is the same for every CCM. From this it immediately follows that the degree of commutativity is independent of the  CCM used.

\subsection*{Indices of centralisers}

Recall from Definition~\ref{defn:x_n}, that for any group $G$ the set $\G{n}$ denotes the set of elements whose centraliser has index at most $n$. Our next aim is to describe the sets $\G{n}$ in a FAF group and relate these to the degree of commutativity. 

We start by looking at the Boolean subring of $\Pset(G)$ generated by the cosets of subgroups. 

\begin{defn} \label{defn:cosets}
    Let $G$ be a group and $\Pset(G)$ the power set of $G$ considered as a Boolean ring.
 We denote by $\cosets$ the subring of $\Pset(G)$ generated by the cosets of subgroups of $G$.
\end{defn}

\begin{rem}
    Note that the identity $gH = (gHg^{-1})g$ means that we do not need to specify whether we intend to use left cosets or right cosets or both, as these all define the same object.
\end{rem}

\begin{lem}
\label{lem:C has a normal form}
    Let $G$ be a group, $\Pset(G)$ the power set of $G$ and $\cosets$ be the subring of generated by the cosets of subgroups of $G$. Then,
    \begin{enumerate}
        \item \label{it:normal-subgroup} $\cosets$ is equal to the additive subgroup generated by the cosets of subgroups of $G$. 
        \item \label{it:normal-expression} In particular, every element of $\cosets$ can be written as a finite sum $H_1g_1 \oplus \cdots \oplus H_kg_k$ for some subgroups $H_i$ of $G$ and elements $g_i$.
    \end{enumerate}
\end{lem}
\begin{proof}
    Note that \ref{it:normal-subgroup} follows from \ref{it:normal-expression}, since the ring has characteristic $2$ and hence the finite sums form a subgroup. To observe \ref{it:normal-expression}, simply note that the intersection of two cosets is either empty or is a coset of the intersection, and hence the finite sums are closed under multiplication. 
\end{proof}

We next analyse some sub-objects of $\cosets$ in order to show that $\cosets$ always admits a unique CCM. 

\begin{defn}
    Given the Boolean subring $\cosets$ of $\Pset(G)$, we let 
    \begin{enumerate}
        \item $\cosetsfin$ denote the additive subgroup generated by the cosets of finite index subgroups of~$G$; and 
        \item $\cosetsinf$ denote the additive subgroup generated by the cosets of infinite index subgroups of~$G$. 
    \end{enumerate}
\end{defn}

We then have:

\begin{prop}
\label{prop:CCM on cosets}
Let $G$ be a group and $\cosets$ be the Boolean subring generated by the cosets of subgroups of $G$. Then,
\begin{enumerate}
    \item \label{prop:fin-subring} $\cosetsfin$ is a subring of $\cosets$; 
    \item \label{prop:inf-ideal} $\cosetsinf$ is an ideal of $\cosets$;
    \item \label{prop:direct-sum} $\cosetsfin \cap \cosetsinf = \{ \varnothing \} $ and $\cosets = \cosetsfin \oplus \cosetsinf$ as abelian groups;
    \item \label{prop:projection} there is a (ring) projection, $\pi\colon \cosets \to \cosetsfin$;
    \item \label{prop:nu-unique} there exists a unique CCM $\nu$ on $\cosetsfin$, and it takes only rational values;
    \item \label{prop:mu-unique} there exists a unique CCM $\mu$ on $\cosets$, given by the composition $\nu \circ \pi$ and only taking rational values.
\end{enumerate}
\end{prop}

\begin{proof}
    We deal with these in turn. 
    \begin{enumerate}
    
        \item This follows since the intersection of two cosets of finite index is either empty or a coset of the intersection, which is again of finite index. 
        
        \item This follows similarly, since the non-empty intersection of a coset of infinite index with an arbitrary coset is a coset of infinite index. 
        
        \item Every non-empty element of $\cosetsfin$ is equal to a disjoint union of cosets of a single finite index subgroup. In contrast, elements of $\cosetsinf$ are finite symmetric differences of cosets of infinite index subgroups, and are not generally disjoint unions of such cosets.
        
        Suppose $X \in \cosetsfin \cap \cosetsinf$ is non-empty. Then $X$ contains a coset of a finite index subgroup and, being in $\cosetsinf$, is contained in a finite union of cosets of infinite index subgroups. This contradicts Theorem~\ref{thm:neumann}. 

        Hence, $\cosetsfin \cap \cosetsinf = \{ \varnothing \}$. The decomposition $\cosets = \cosetsfin \oplus \cosetsinf$ as abelian groups now follows from Lemma~\ref{lem:C has a normal form}.

        \item This is immediate from \ref{prop:inf-ideal} and \ref{prop:direct-sum}. 
        
        \item Since every element $X$ of $\cosetsfin$ is a disjoint union of cosets of a single finite index subgroup $H = H(X)$, the values of a CCM $\nu$ on elements of this subring are uniquely determined and rational. The only thing to check is that the values of $\nu$ do not depend on the choices of the $H(X)$, and that $\nu$ is finitely additive; this comes down to saying that if $K \leq H \leq G$ are finite index subgroups of $G$, then every coset $gH$ is a disjoint union of $[H:K]$  cosets of $K$, and hence the CCM gives the same value whether we think of $gH$ as an $H$-coset or a union of $K$-cosets.
        
        \item Since $\pi$ is a ring homomorphism, it is readily seen that $\mu$ is a CCM on $\cosets$, and it takes only rational values since so does $\nu$. Moreover, if $\mu'$ were another CCM on $\cosets$ then we can use \ref{prop:direct-sum} to write any $X \in \cosets$ as $X=X_{\textup{fin}} \oplus X_{\textup{inf}}$ uniquely, where $X_{\textup{fin}} \in \cosetsfin$ and $X_{\textup{inf}} \in \cosetsinf$. From finite additivity we get that
        \[
        \mu'(X_{\textup{fin}}) = \mu'(X_{\textup{fin}}) - \mu'(X_{\textup{inf}}) \leq \mu'(X_{\textup{fin}} \oplus X_{\textup{inf}}) \leq \mu'(X_{\textup{fin}}) + \mu'(X_{\textup{inf}}) = \mu'(X_{\textup{fin}}),
        \]
        from which it follows that $\mu'=\mu$ by \ref{prop:nu-unique}. \qedhere
    \end{enumerate}
\end{proof}

We now aim to show that the sets $U_m \coloneqq \G{m} - \G{m-1}$ are elements of $\cosets$ in an arbitrary FAF group $G$.  The idea is to take a FAF group $G$ with its normal finite index subgroup $H \sgpnfin G$ such that $N \coloneqq [H,H]$ is central in $H$ and finite (the existence of such an $H$ is guaranteed by Remark~\ref{rem:FCAF}), and consider the inclusions $C_G(g) \leq \overline{C}(gN) \leq G$, where $\overline{C}(gN)$ is the preimage of $C_{G/N}(gN)$ under the quotient map $G \to G/N$.  The index $[\overline{C}(gN):C_G(g)]$ is studied in Lemma~\ref{lem:FAF-technical}, and the index $[G:\overline{C}(gN)]$ in Proposition~\ref{prop:vab} (which applies here since $G/N$ is virtually abelian).  These results are then combined in Corollary~\ref{cor:FAF-Um} to show that $U_m \in \cosets$.

\begin{prop} \label{prop:vab}
    Let $G$ be a virtually abelian group, and let $\mathcal{K}$ be the set of all centralisers of elements $g \in \Gfin$.  Then $\mathcal{K}$ is finite, and we have
    \[
    S_K \coloneqq \{ g \in G : C_G(g) = K \} \in \cosets \quad \text{for all } K \in \mathcal{K}.
    \]
\end{prop}

\begin{proof}
    Let $A \leq G$ be a finite index abelian subgroup.  Then for all $g \in G$ and $a \in A$, we have $C_A(g) = C_A(ga)$: indeed, an arbitrary element $b \in A$ commutes with $a$ since $A$ is abelian, and therefore $b$ commutes with $g$ if and only if it commutes with $ga$.  In particular, since $A$ has finite index in $G$, there are only finitely many subgroups of the form $C_A(g)$ when $g$ ranges over the elements of $G$.  Note that if $g \in \Gfin$ then $[G:C_A(g)] = [G:A][A:C_A(g)] < \infty$, and therefore $N_0 = \bigcap_{g \in \Gfin} C_A(g)$ has finite index in $G$.

    Let $N = \Core_G(N_0) \unlhd G$, and note that $N$ has finite index in $G$.  Then for every $g \in \Gfin$, the centraliser $C_G(g)$ contains $N$, and therefore is the preimage in $G$ of a subgroup of $G/N$.  Since $G/N$ is finite it only has finitely many subgroups, implying that the set $\mathcal{K}$ is finite.

    Now note that given an element $g \in \Gfin$ and a subgroup $K \in \mathcal{K}$, we have $C_G(g) = K$ if and only if $g$ centralises $K$ but does not centralise $K_0$ for any $K \lneq K_0 \in \mathcal{K}$.  This shows that
    \[
    S_K = C_G(K) - \bigcup_{K \lneq K_0 \in \mathcal{K}} C_G(K_0).
    \]
    In particular, since $\cosets$ is closed under finite unions and set differences, we have $S_K \in \cosets$, as required.
\end{proof}

The following Lemma is technical.  Given a finite index subgroup $K \sgpfin G/N$, we consider its preimage $\overline{K} \leq G$ as well as the subgroup $\overline{C}_K \leq G$ of elements that centralise $\overline{K}$ modulo $N$.  We show that given any $g \in \overline{C}_K$, the set of commutators $M_g \coloneqq [g,H \cap \overline{K}] \subseteq N$ is a subgroup, and for any subgroup $M \leq N$, the set $L_M \coloneqq \{ g \in \overline{C}_K : M_g \leq M \}$ is a subgroup as well.  This allows us to control the set of elements for which $M_g$ has a certain cardinality.  Moreover, we show that given $g \in L_M$, the coset $[g,h]M$ for $h \in \overline{K}$ depends only on the coset $h(H \cap \overline{K})$, which allows us to have control over the number $|[g,\overline{K}]| / |M_g|$.  Combining these results, we get a description of the set of elements $g \in \overline{C}_K$ for which $[g,\overline{K}]$ has some fixed cardinality $m \in \N$.  But this cardinality is the same as the cardinality of the $\overline{K}$-conjugacy class of $g$, i.e.\ equal to the index $[\overline{K} : C_{\overline{K}}(g)]$.  The result then follows by noting that if $C_{G/N}(gN) = K$, then we have $C_G(g) \subseteq \overline{K}$ and therefore $C_{\overline{K}}(g) = C_G(g)$.

\begin{lem} \label{lem:FAF-technical}
    Let $G$ be a group that has a finite normal subgroup $N \unlhd G$ and a finite index normal subgroup $H \sgpnfin G$ containing $N$ such that $N$ is central in $H$ and $H/N$ is abelian.  Let $K \sgpfin G/N$, let $m \in \N \cup \{\infty\}$, and write $\pi_N\colon G \to G/N$ for the quotient map.  Then we have
    \[
    S_{K,m} \coloneqq \{ g \in \overline{S}_K : [\pi_N^{-1}(K) : C_G(g)] = m \} \in \cosets,
    \]
    where $\overline{S}_K = \{ g \in G \colon C_{G/N}(gN) = K \}$, and $S_{K,m} = \varnothing$ for all $m > |N|$.
\end{lem}

\begin{proof}
    We write $\overline{K} \coloneqq \pi_N^{-1}(K)$ and $\overline{C}_K \coloneqq \pi_N^{-1}(C_{G/N}(K))$; note that $\overline{S}_K \subseteq \overline{C}_K$ and that since $G/N$ is virtually abelian, we have $\overline{S}_K \in \cosets$ by Proposition~\ref{prop:vab}.  Given a subgroup $M \leq N$, we let
    \[
    L_M \coloneqq \{ g \in \overline{C}_K : [g,H \cap \overline{K}] \subseteq M \},
    \]
    where $[g,H \cap \overline{K}] = \{ [g,h] : h \in H \cap \overline{K} \}$.  The proof then proceeds in the following steps:
    \begin{enumerate}
        \item \label{it:faf-ghk} showing that $[g,H \cap \overline{K}]$ is a subgroup of $N$ for every $g \in \overline{C}_K$;
        \item \label{it:faf-gk} showing that $[g,\overline{K}]$ is a union of cosets of $[g,H \cap \overline{K}]$ in $N$ for every $g \in \overline{C}_K$;
        \item \label{it:faf-lm} showing that $L_M$ is a subgroup of $G$;
        \item \label{it:faf-phig} showing that for any $g \in L_M$, the function (which is not necessarily a group homomorphism) $\phi_g\colon \overline{K} \to N/M$, defined by $\phi_g(h) = [g,h]M$, factors through the quotient map $\overline{K} \xrightarrow{q} \overline{K}/(H \cap \overline{K})$, i.e.\ $\phi_g = \Phi_g \circ q$ for some function $\Phi_g\colon \overline{K}/(H \cap \overline{K}) \to N/M$;
        \item \label{it:faf-psi} showing that the function $\Psi\colon L_M \to (N/M)^{\overline{K}/(H \cap \overline{K})}$ defined by $g \mapsto \Phi_g$ is a group homomorphism.
    \end{enumerate}

    After proving all these points, the proof can be finalised as follows.  Given a subgroup $M \leq N$, we set $T_M = \{ g \in \overline{C}_K : [g,H \cap \overline{K}] = M \}$.  It then follows from point~\ref{it:faf-ghk} that we have
    \[
    T_M = L_M - \bigcup_{M_0 \lneq M} L_{M_0}.
    \]
    By point \ref{it:faf-lm} we have $L_{M_0} \in \cosets$ for all $M_0 \leq M$; in particular, since $M$ is finite (and so has finitely many subgroups) and since $\cosets$ is closed under finite unions and set differences, it follows that $T_M \in \cosets$.

    Now given any $g \in \overline{S}_K$, note that $\overline{K} = \pi_N^{-1}(K)$ is precisely the set of elements $h \in G$ such that $[g,h] \in N$, and in particular $C_G(g) \leq \overline{K}$.  On the other hand, $\overline{K}$ acts transitively on the set of $\overline{K}$-conjugacy classes of $g$, $\{ h^{-1}gh : h \in \overline{K} \} = g[g,\overline{K}]$, and $C_G(g)$ is precisely the stabiliser of $g$ under this action.  In particular, we have $[\overline{K} : C_G(g)] = |[g,\overline{K}]|$.  Since $[g,\overline{K}] \subseteq N$, it also follows that $S_{K,m} = \varnothing$ for all $m > |N|$.

    Note that if $g \in T_M$ for some $M \leq N$, then it follows from point~\ref{it:faf-gk} and the definition of $\phi_g$ that we have $|[g,\overline{K}]| = |M| \cdot |\phi_g(\overline{K})|$.  In particular, if given any $r \in \N$ we set $\mathcal{F}_r \coloneqq \{ g \in L_M : |\phi_g(\overline{K})| = r \}$, we then have
    \[
    S_{K,m} = \overline{S}_K \cap \bigsqcup_{\substack{M \leq N \\ m/|M| \in \N}} T_M \cap \mathcal{F}_{m/|M|}.
    \]
    Now by point~\ref{it:faf-psi}, it follows that for each $r$, the set $\mathcal{F}_r$ is a union of cosets of $\ker(\Psi)$.  Since $H$ has finite index in $G$ and since $N$ is finite, it follows that both $\overline{K}/(H \cap \overline{K})$ and $N/M$ are finite (for all $M \leq N$), and therefore $(N/M)^{\overline{K}/(H \cap \overline{K})}$ is also finite.  We therefore have that $\ker(\Psi)$ is a finite index subgroup of $L_M$ and so $\mathcal{F}_r$ is a finite union of cosets of $\ker(\Psi)$ for each $r \in \N$; thus $\mathcal{F}_r \in \cosets$.  Since $\cosets$ is closed under finite unions and intersections, and since $T_M \in \cosets$ for all $M \leq N$ (by the argument above) and $\overline{S}_K \in \cosets$ (by Proposition~\ref{prop:vab}), it then follows that $S_{K,m} \in \cosets$, as required.

    It thus remains to prove the points \ref{it:faf-ghk}--\ref{it:faf-psi}.

    \begin{enumerate}
    
        \item Given $h_1,h_2 \in H \cap \overline{K}$, note that $[g,h_1] \in N$ since $g \in \overline{C}_K$ and $h_1 \in \overline{K}$, and therefore $[g,h_1]^{h_2} = [g,h_1]$ since $N$ is central in $H$.  We then have, using commutator calculus and the fact that $N$ is abelian,
        \[
        [g,h_1h_2] = [g,h_2] [g,h_1]^{h_2} = [g,h_2] [g,h_1] = [g,h_1] [g,h_2],
        \]
        and so the map $H \cap \overline{K} \to N$ sending $h \mapsto [g,h]$ is a group homomorphism.  It follows that the image of this map, $[g,H \cap \overline{K}]$, is a subgroup of $N$, as required.

        \item \label{it:faf-gk-pf} This is similar to the previous point: it is enough to show that given $k \in \overline{K}$ we have $[g,k][g,H \cap \overline{K}] \subseteq [g,\overline{K}]$---that is, $[g,k][g,h] \in [g,\overline{K}]$ for every $h \in H \cap \overline{K}$.  But we have $[g,k] \in N$ since $g \in \overline{C}_k$ and therefore $[g,k]^h = [g,k]$ since $N$ is central in $H$.  We thus have
        \[
        [g,k][g,h] = [g,k]^h [g,h] = [g,h] [g,k]^h = [g,kh] \in [g,\overline{K}],
        \]
        as required.

        \item Clearly $1 \in L_M$.  We first claim that $g^{-1} \in L_M$ for a given $g \in L_M$.  Indeed, for any $h \in H \cap \overline{K}$ we have
        \[
        [g^{-1},h] = [h,g]^{g^{-1}} = [h^{g^{-1}},g] = [g,h^{g^{-1}}]^{-1}.
        \]
        On the other hand, note that $h^{g^{-1}} \in H$ since $H$ is normal in $G$, and $h^{g^{-1}} \in \overline{K}$ since the fact that $g \in \overline{C}_K$ implies that $\overline{K}$ is invariant under conjugation by $g$.  Thus we have $[g^{-1},h] \in [g,H \cap \overline{K}]^{-1} \subseteq M$ for all $h \in H \cap \overline{K}$, showing that $g^{-1} \in L_M$, as claimed.

        We now claim that $g_1g_2 \in L_M$ for given $g_1,g_2 \in L_M$.  Indeed, for any $h \in H \cap \overline{K}$ we have, by commutator calculus,
        \[
        [g_1g_2,h] = [g_1,h]^{g_2} [g_2,h] = [g_1,h] [[g_1,h],g_2] [g_2,h] = [g_1,h] [g_2,[g_1,h]]^{-1} [g_2,h].
        \]
        Since $[g_1,h] \in N \subseteq H \cap \overline{K}$, it follows that all three terms in the equation above are elements of $M$ for any $h \in H \cap \overline{K}$ and so $[g_1g_2,h] \in M$, implying that $g_1g_2 \in L_M$, as claimed.

        \item It is enough to show that we have $\phi_g(k) = \phi_g(kh)$ for all $k \in \overline{K}$ and $h \in H \cap \overline{K}$.  But this follows from the computation in point~\ref{it:faf-gk-pf} above: indeed, we have
        \[
        \phi_g(kh) = [g,kh]M = [g,k][g,h]M = [g,k]M = \phi_g(k),
        \]
        since the fact that $g \in L_M$ implies that $[g,h] \in M$.

        \item Given $g_1,g_2 \in L_M$ and $k \in \overline{K}$, note that we have
        \[
        [g_2,[g_1,k]] \in [g_2,N] \subseteq [g_2,H \cap \overline{K}] \subseteq M.
        \]
        We thus have, by commutator calculus,
        \begin{align*}
             \phi_{g_1g_2}(k) &= [g_1g_2,k]M = [g_1,k]^{g_2}[g_2,k]M = [g_1,k] [g_2,[g_1,k]]^{-1} [g_2,k] M \\ &= [g_1,k] [g_2,k] M = \phi_{g_1}(k) \phi_{g_2}(k) = (\phi_{g_1}\phi_{g_2})(k),
        \end{align*}
        and thus $\Psi$ is a group homomorphism, as required. \qedhere
        
    \end{enumerate}
\end{proof}

\begin{cor} \label{cor:FAF-Um}
    Let $G$ be a FAF group, and let $m \in \N \cup \{\infty\}$.  Then we have
    \[
    U_m \coloneqq \{ g \in G : [G:C_G(g)] = m \} \in \cosets,
    \]
    and only finitely many of the $U_m$ are non-empty.
\end{cor}

\begin{proof}
    By Remark~\ref{rem:FCAF}, we may assume that $G$ has subgroups $N$ and $H$ as in the statement of Lemma~\ref{lem:FAF-technical}.  Note that the group $G/N$ is virtually abelian.

    We first observe that only finitely many of the $U_m$ are non-empty.  Indeed, note that by Proposition~\ref{prop:vab}, only finitely many finite index subgroups of $G/N$ can appear as centralisers of elements of $G/N$.  On the other hand, it follows from Lemma~\ref{lem:FAF-technical} that for any $g \in \Gfin$ we have $[\pi_N^{-1}(C_{G/N}(gN)) : C_G(g)] \leq |N|$.  Combining these two facts, we get an upper bound on the finite indices of centralisers of elements in $G$, as required.

    In particular, it is enough to show that $U_m \in \cosets$ for all $m \in \N$, as then $U_\infty$ will be the complement of a finite union of elements of $\cosets$ and so will also belong to $\cosets$.

    Now by Proposition~\ref{prop:vab}, there is a finite set $\mathcal{K}$ of finite index subgroups of $G/N$ that occur as centralisers of elements.  Moreover, by Lemma~\ref{lem:FAF-technical}, for any $K \in \mathcal{K}$ and any $r \in \N$ we have
    \[
    S_{K,r} \coloneqq \{ g \in G : C_{G/N}(gN) = K, [\pi_N^{-1}(K):C_G(g)] = r \} \in \cosets.
    \]
    In particular, if we denote by $i_K$ the index $[G:\pi_N^{-1}(K)]$ for any $K \in \mathcal{K}$, it then follows that
    \[
    U_m = \bigsqcup_{\substack{K \in \mathcal{K} \\ m/i_K \in \N}} S_{K,m/i_K},
    \]
    and so $U_m \in \cosets$ for any $m \in \N$, as required.
\end{proof}

\subsection*{Degree of commutativity}

We are now in a position to prove the degree of commutativity is independent of the CCM used for a FAF group. 

\begin{prop} \label{prop:CCM-dc-independent}
    Let $\mu$ and $\nu$ be CCMs on a FAF group $G$.  Then $\dc[\mu](G) = \dc[\nu](G) \in \Q$.
\end{prop}

\begin{proof}
    By Corollary~\ref{cor:FAF-Um}, we have $G = \left( \bigsqcup_{m=1}^M U_m \right) \sqcup U_\infty$ for some $M \in \N$, where
    \[
    U_m = \{ g \in G : [G:C_G(g)] = m \} \in \cosets.
    \]
    In particular, we have
    \begin{align*}
        \dc[\mu](G) &= \int \mu(C_G(g)) \,\mathrm{d}\mu(g) = \int \frac{1}{[G:C_G(g)]} \,\mathrm{d}\mu(g)
        \\ &= \left( \sum_{m=1}^M \int \frac{1}{m} \mathbbm{1}_{U_m} \,\mathrm{d}\mu \right) + \int 0 \cdot \mathbbm{1}_{U_\infty} \,\mathrm{d\mu} = \sum_{m=1}^M \frac{\mu(U_m)}{m},
    \end{align*}
    and similarly $\dc[\nu](G) = \sum_{m=1}^M \frac{\nu(U_m)}{m}$.  But for $1 \leq m \leq M$ we have $U_m \in \cosets$ and therefore $\mu(U_m) = \nu(U_m) \in \Q$ by Lemma~\ref{prop:CCM on cosets}\ref{prop:mu-unique}, implying that $\dc[\mu](G) = \dc[\nu](G) \in \Q$, as required.
\end{proof}

\begin{thm} \label{thm:dc-CCM-nonzero}
    \ThmDCCCM
\end{thm}

\begin{proof}
    Suppose that $G$ is FAF. In particular, $G$ has a finite index normal subgroup $H \unlhd G$ such that $N \coloneqq [H,H]$ is finite and central in $H$ (see Remark~\ref{rem:FCAF}).  Since $H$ is $2$-step nilpotent, it follows from commutator calculus that the map $\psi_g\colon H \to N$ defined by $\psi_g(x) = [g,x]$ is a group homomorphism for any $g \in H$.  In particular, we have $[H:C_H(g)] = [H:\ker(\psi_g)] = |\psi_g(H)| \leq |N|$ for every $g \in H$, implying that
    \begin{align*}
    \dc[\mu](G) &= \int \mu(C_G(g)) \,\mathrm{d}\mu = \int [G:C_G(g)]^{-1} \,\mathrm{d}\mu \geq \int [G:C_G(g)]^{-1}\mathbbm{1}_H \,\mathrm{d}\mu \\ &\geq \int [G:C_H(g)]^{-1}\mathbbm{1}_H \,\mathrm{d}\mu \geq (|N| \cdot [G:H])^{-1} \int \mathbbm{1}_H \,\mathrm{d}\mu = \frac{1}{|N| \cdot [G:H]^2}
    \end{align*}
    and thus $\dc[\mu](G) > 0$, as required.

    Suppose now instead that $\dc[\mu](G) > 0$. By Lemma~\ref{lem:bigcentraliser}, there exists $m \in \N$ such that $\mu(\G{m}) > 0$.  It then follows by Proposition~\ref{prop:dc-faf} that $G$ is FAF.

    Finally, the rationality and independence from the choice of $\mu$ follows from Proposition~\ref{prop:CCM-dc-independent} if $G$ is FAF, and from the fact that $\dc[\mu](G) = 0$ otherwise.
\end{proof}

\section{Residually finite groups}
\label{sec:RF}

Having defined the residually finite degree of commutativity in Definition~\ref{defn:dc for RF} it is natural to see whether this quantity agrees with the analogous number obtained using the Haar measure. 

\begin{defn}
    For a residually finite group $G$, the \emph{profinite completion} $\widehat{G}$ is the inverse limit of the directed system of finite quotients of $G$:
    \[
    \widehat{G} \coloneqq \varprojlim \{ G/N : N \sgpnfin G \}.
    \]
    The diagonal embedding $g \mapsto (gN)$ embeds $G$ as a dense subgroup into $\widehat{G}$, and we will therefore identify $G$ with a subgroup of $\widehat{G}$.  The group $\widehat{G}$, with the inverse limit topology (induced by the discrete topology on the finite quotients of $G$), is a compact Hausdorff totally disconnected topological group and has a unique left invariant inner-regular and outer-regular Borel probability measure $\widehat{\mu}$, called the \emph{Haar measure}.
\end{defn}

\begin{thm} \label{thm:dc-res-finite}
    \ThmDCRF
\end{thm}

\begin{proof}
The last statement of this Theorem is simply Lemma~\ref{lem:dcRF0}, so it suffices to prove the displayed equalities.

    Note that $\Comm(\widehat{G})$ is closed in $\widehat{G}^2$, since the commutator map $(g,h) \mapsto [g,h]$ is continuous and since $\{1\}$ is closed in $\widehat{G}$.  As $\Comm(G) = G^2 \cap \Comm(\widehat{G}) \subseteq \Comm(\widehat{G})$, it follows that $\overline{\Comm(G)} \subseteq \Comm(\widehat{G})$.  This implies that $\widehat\mu(\overline{\Comm(G)}) \leq \widehat\mu(\Comm(\widehat{G}))$.
    
    On the other hand, if $N \unlhd G$ is a normal finite index subgroup of $G$, then the preimage of $\Comm(G/N)$ in $\widehat{G}^2$ is equal to $\{ (g,h) \in \widehat{G}^2 : [g,h] \in \overline{N} \}$ and therefore contains $\Comm(\widehat{G})$, and the Haar measure of this preimage is exactly equal to $\dc(G/N)$.  It follows that $\widehat\mu(\Comm(\widehat{G})) \leq \dc(G/N)$ for all $N \sgpnfin G$, and thus $\widehat\mu(\Comm(\widehat{G})) \leq \dcRF(G)$.  In particular, we have
    \[
    \widehat\mu(\overline{\Comm(G)}) \leq \widehat\mu(\Comm(\widehat{G})) \leq \dcRF(G).
    \]
    Similarly, $\dc[\nu](G) \leq \dcRF(G)$ by Lemma~\ref{lem: dcCCM less that dcRF}.  It is hence enough to show that $\dcRF(G) \leq \widehat\mu(\overline{\Comm(G)})$ and $\dcRF(G) \leq \dc[\nu](G)$.

    Now if $\dcRF(G) = 0$, then the result follows immediately, so we may assume that $\dcRF(G) > 0$.  Then, by Lemma~\ref{lem:dcRF0}, $G$ is virtually abelian, and so contains a normal finite index abelian subgroup $A \unlhd G$.  Recall that the set $\Gfin$ of elements of $G$ whose centraliser has finite index forms a normal subgroup, and note that $\Gfin = \{ g \in G : [G:C_A(g)] < \infty \}$ as $A$ has finite index in $G$.  We aim to construct:
    \begin{enumerate}
        \item a finite index normal subgroup $M \sgpnfin G$ that is contained in $\Gfin$ such that we have $\Comm(G) \cap \Gfin^2 = \{ (g,h) \in \Gfin^2 : [g,h] \in M \}$; and
        \item for any $\varepsilon > 0$, a finite index normal subgroup $N_\varepsilon \sgpnfin G$ contained in $M$ such that $(G/N_\varepsilon)^2$ contains at most $\varepsilon|G/N_\varepsilon|^{2}$ commuting pairs outside of $(\Gfin/N_\varepsilon)^2$.
    \end{enumerate}
    Now having constructed $N_\varepsilon$, note that since $N_\varepsilon \subseteq M \subseteq \Gfin$, we have that $\Comm(G) \cap \Gfin^2$ is a union of cosets of~$N_\varepsilon^2$ in $G^2$, and that the image of $\Comm(G) \cap \Gfin^2$ in $(G/N_\varepsilon)^2$ is precisely $\Comm(G/N_\varepsilon) \cap (\Gfin/N_\varepsilon)^2$.  It follows that $\overline{\Comm(G)} \subseteq \widehat{G}^2$ contains the preimage of $\Comm(G/N_\varepsilon) \cap (\Gfin/N_\varepsilon)^2$ under the surjection $\widehat{G}^2 \to (G/N_\varepsilon)^2$, and therefore
    \[
    \widehat\mu(\overline{\Comm(G)}) \geq \frac{|\Comm(G/N_\varepsilon) \cap (\Gfin/N_\varepsilon)^2|}{|G/N_\varepsilon|^2}.
    \]
    Similarly, $\Comm(G)$ contains the preimage of $\Comm(G/N_\varepsilon) \cap (\Gfin/N_\varepsilon)^2$ under the surjection $G^2 \to (G/N_\varepsilon)^2$, and therefore
    \[
    \dc[\nu](G) \geq \frac{|\Comm(G/N_\varepsilon) \cap (\Gfin/N_\varepsilon)^2|}{|G/N_\varepsilon|^2}.
    \]
    
    On the other hand, we have $\frac{|\Comm(G/N_\varepsilon) - (\Gfin/N_\varepsilon)^2|}{|G/N_\varepsilon|^2} \leq \varepsilon$ by the choice of $N_\varepsilon$, and therefore
    \[
    \dc(G/N_\varepsilon) - \frac{|\Comm(G/N_\varepsilon) \cap (\Gfin/N_\varepsilon)^2|}{|G/N_\varepsilon|^2} \leq \varepsilon.
    \]
    It follows that $\widehat\mu(\overline{\Comm(G)}) \geq \dc(G/N_\varepsilon) - \varepsilon \geq \dcRF(G) - \varepsilon$, and similarly $\dc[\nu](G) \geq \dcRF(G)-\varepsilon$.  But $\varepsilon > 0$ was arbitrary, implying that we have $\widehat\mu(\overline{\Comm(G)}) \geq \dcRF(G)$ and $\dc[\nu](G) \geq \dcRF(G)$, as required.

    It remains to construct the subgroups $M$ and $N_\varepsilon$ as above.

    \begin{enumerate}
    
        \item Note that $[G:\Gfin] < \infty$ (since $A \subseteq \Gfin$) and set
        \[
        M_0 \coloneqq \bigcap_{g \in \Gfin} C_A(g).
        \]
        Note that $M_0 \unlhd G$ (since $\Gfin \unlhd G$ and $A \unlhd G$).  Moreover, given any $g \in G$ and $a \in C_A(g)$, a given element $b \in A$ commutes with $a$, and so it commutes with $g$ if and only if it commutes with $ga$, showing that $C_A(g) = C_A(ga)$.  Thus $C_A(g)$ is constant as $g$ ranges over a coset of $A$ in $G$; as $[G:A] < \infty$, it follows that $M_0$ is the intersection of \emph{finitely many} finite index subgroups of $G$ and therefore $[G:M_0] < \infty$.

        Now let $\mathcal{T}$ be a transversal of $M_0$ in $\Gfin$ containing $1$ (note that $M_0 \subseteq A \subseteq \Gfin$), and let $C = \{ [g,h] : (g,h) \in \mathcal{T}^2 \}$.  Then $C$ is a finite subset of $G$ containing~$1$.  Since $G$ is residually finite, there exists a finite index normal subgroup $M_1 \unlhd G$ such that the quotient map $G \to G/M_1$ is injective on $C$, and in particular $C \cap M_1 = \{1\}$.
        
        We set $M = M_0 \cap M_1$; by construction, $M$ is a normal finite index subgroup of $G$ contained in $\Gfin$.  Given any $g,h \in \mathcal{T}$ and $a,b \in M_0$, note that $[ga,b] = 1 = [a,h]$ by the definition of $M_0$, implying that
        \[
        [ga,hb] = [ga,b] [ga,h]^b = [ga,h]^b = [g,h]^{ab} [a,h]^b = [g,h]^{ab}.
        \]
        In particular, since $M \unlhd G$, we have $[ga,hb] \in M$ if and only if $[g,h] \in M$, which happens if and only if $[g,h] = 1$ by the choice of $M_1$.  It follows that for any $g,h \in \Gfin$, we have either $[g,h] = 1$ or $[g,h] \notin M$, and therefore $\Comm(G) \cap \Gfin^2 = \{ (g,h) \in \Gfin^2 : [g,h] \in M \}$, as required.

        \item Note that given $g \in G$ and $a,b \in A$, we have
        \[
        [ga,b] = [g,b]^a [a,b] = [g,b]
        \]
        since $A$ is normal and abelian, and in particular $[g,A] = [ga,A]$.  Thus $[g,A]$ is constant as $g$ ranges over elements in a coset of $A$ in $G$, and so, since $[G:A] < \infty$, there are only finitely many subsets of $G$ of the form $[g,A]$, for $g \in G$.  Note also that we have $|[g,A]| = |g[g,A]| = |g^A| = [G:C_A(g)]$ by the orbit-stabiliser theorem, and in particular $|[g,A]| < \infty$ if and only if $g \in \Gfin$.

        Let $n \in \N$ be such that $2/n \leq \varepsilon$.  By the paragraph above, there exists a finite subset $S \subseteq G$ such that for each $g \in G - \Gfin$, the subset $S$ contains at least $n$ commutators of the form $[g,a]$ for $a \in A$.  Since $G$ is residually finite, there exists a finite index normal subgroup $N_\varepsilon \sgpnfin G$ such that the quotient map $G \to G/N_\varepsilon$ is injective on $S$; by replacing $N_\varepsilon$ with $N_\varepsilon \cap A \cap M$ if necessary, we may further assume that $N_\varepsilon \subseteq A$ and $N_\varepsilon \subseteq M$.

        It then follows that for any $g \in G - \Gfin$, the subset $[gN_\varepsilon,A/N_\varepsilon] = [g,A]N_\varepsilon/N_\varepsilon \subseteq G/N_\varepsilon$ contains at least $n$ elements, and therefore
        \[
        [G/N_\varepsilon : C_{G/N_\varepsilon}(gN_{\varepsilon})] = |[gN_\varepsilon,G/N_\varepsilon]| \geq |[gN_\varepsilon,A/N_\varepsilon]| \geq n
        \]
        by the orbit-stabiliser theorem.  Now note that we have
        \begin{align*}
            \Comm(G/N_\varepsilon) - (\Gfin/N_\varepsilon)^2 &= \left( \bigcup_{g \in G-\Gfin} \{gN_\varepsilon\} \times C_{G/N_\varepsilon}(gN_{\varepsilon}) \right) \\ &\cup \left( \bigcup_{h \in G-\Gfin} C_{G/N_\varepsilon}(hN_{\varepsilon}) \times \{hN_\varepsilon\} \right),
        \end{align*}
        and therefore, by symmetry,
        \begin{align*}
            \left| \Comm(G/N_\varepsilon) - (\Gfin/N_\varepsilon)^2 \right| &\leq 2 \sum_{gN_\varepsilon \in (G-\Gfin)/N_\varepsilon} |C_{G/N_\varepsilon}(gN_{\varepsilon})|
            \\ &\leq 2 \cdot |(G-\Gfin)/N_\varepsilon| \cdot \frac{|G/N_\varepsilon|}{n} 
            \\ &\leq \frac{2|G/N_\varepsilon|^2}{n} \leq \varepsilon|G/N_\varepsilon|^2,
        \end{align*}
        as required. \qedhere
        
    \end{enumerate}
\end{proof}

\section{Amenable groups}
\label{sec:Am}

We note that our construction to date has been to look at a CCM on a group $G$ and extend to a product mean on $G^2$ from which we define a degree of commutativity and deduce structural results on $G$. However, one could instead look at an arbitrary CCM on $G^2$ and, as we shall see---Example~\ref{ex:heisenberg}---this can produce a different answer. 

However we can also ask the same questions for invariant means and amenable groups. Namely, what can we say about a degree of commutativity for an amenable group which is defined using a non-product mean on $G^2$? In particular, one should bear in mind that FAF groups are all amenable. 

Our result here is that the positivity of these quantities corresponds to the group being FAF in the amenable case, even if one doesn't use the product mean.

\subsection*{Non-product CCMs on products}

We recall, Definition~\ref{defn:x_n}, that $\G{n}$ consists of all the elements of $G$ whose centraliser has index at most $n$, and that $\Gfin$ is a normal subgroup of $G$ that is the union of all the $\G{n}$.

\begin{lem} \label{lem:small-CGhi}
    Suppose that either $\G{n} \neq \Gfin$ for any $n \in \N$, or $[G:\Gfin] = \infty$.  Then for every $n \in \N$ and every $\varepsilon > 0$, there exist $h_1,\ldots,h_n \in G$ such that
    \[
    \sum_{1 \leq i < j \leq n} [G:C_G(h_i^{-1}h_j)]^{-1} < \varepsilon.
    \]
\end{lem}

\begin{proof}
    Suppose first that $[G:\Gfin] = \infty$.  We then choose any elements $h_1,\ldots,h_n$ that represent $n$ distinct cosets of $\Gfin$ in $G$.  This implies that $h_i^{-1}h_j \notin \Gfin$ and therefore $[G:C_G(h_i^{-1}h_j)] = \infty$ for all $1 \leq i < j \leq n$, implying that $\sum_{1 \leq i < j \leq n} [G:C_G(h_i^{-1}h_j)]^{-1} = 0 < \varepsilon$, as required.

    Suppose now that $\G{n} \neq \Gfin$ for any $n \in \N$.  We will then construct $h_1,\ldots,h_n \in \Gfin$ such that
    \[
    c_k \coloneqq \sum_{1 \leq i < j \leq k} [G:C_G(h_i^{-1}h_j)]^{-1} \leq \varepsilon(1-2^{1-k})
    \]
    for $1 \leq k \leq n$, inductively as follows.  We pick an arbitrary $h_1 \in \Gfin$.  Having constructed $h_1,\ldots,h_k$ for some $k \geq 1$, let $N_k \in \N$ be such that $N_k \geq \frac{2^kkM_k}{\varepsilon}$, where $M_k = \max \{ [G:C_G(h_i)] : 1 \leq i \leq k \}$.  Since $\G{N_k} \neq \Gfin$, there exists $h_{k+1} \in \Gfin$ such that $[G:C_G(h_{k+1})] > N_k$.  We then have
    \[
    N_k < [G:C_G(h_{k+1})] \leq [G:C_G(h_i) \cap C_G(h_i^{-1}h_{k+1})] \leq M_k \cdot [G:C_G(h_i^{-1}h_{k+1})]
    \]
    for $1 \leq i \leq k$, implying that
    \[
    [G:C_G(h_i^{-1}h_{k+1})] \geq \frac{N_k}{M_k} \geq \frac{2^kk}{\varepsilon}.
    \]
    In particular, it follows that
    \[
    c_{k+1} = c_k + \sum_{i=1}^k [G:C_G(h_i^{-1}h_{k+1})]^{-1} \leq \varepsilon(1-2^{1-k}) + k \cdot \frac{\varepsilon}{2^kk} = \varepsilon(1-2^{-k}),
    \]
    as required.
\end{proof}

\begin{cor} \label{cor:am-mu}
    Suppose that $G$ is amenable, and let $\mu$ be any invariant mean on $G^2$.  If $\mu(\Comm(G)) > 0$, then there exists $n \in \N$ such that $\mu(\G{n} \times G) > 0$.
\end{cor}

\begin{proof}
    Suppose, for contradiction, that $\mu(\G{n} \times G) = 0$ for all $n \in \N$.  If we have $\G{n} = \Gfin$ for some $n \in \N$, then we have $\mu(\Gfin \times G) = 0$ and therefore, as $\Gfin \times G$ is a subgroup of $G^2$ and $\mu$ is an invariant mean, we have $[G:\Gfin] = \infty$.  Thus either way, the assumption of Lemma~\ref{lem:small-CGhi} is satisfied.
    
    Taking $\varepsilon = \frac{\mu(\Comm(G))}{2}$ and taking $n \in \N$ such that $n > \frac{1}{\varepsilon}$, we then get elements $h_1,\ldots,h_n \in G$ such that
    \[
    \sum_{1 \leq i < j \leq n} \mu(C_G(h_i^{-1}h_j) \times G) < \varepsilon.
    \]
    In particular, taking $Y = \bigcup_{1 \leq i < j \leq n} C_G(h_i^{-1}h_j) \subseteq G$, we then have $\mu(Y \times G) < \varepsilon$ and therefore
    \[
    \mu(\Comm(G) - Y \times G) \geq \mu(\Comm(G)) - \mu(Y \times G) > 2\varepsilon-\varepsilon = \varepsilon.
    \]
    On the other hand, for $1 \leq i \leq n$ define $Y_i = (1,h_i)(\Comm(G) - Y \times G)$.  Then $\mu(Y_i) \geq \varepsilon$ for all $i$ since $\mu$ is left-invariant.  However, given $(g,h) \in G^2 - Y \times G$ we have
    \[
    (g,h) \in Y_i \iff  [g,h_i^{-1}h] = 1 \iff h \in h_iC_G(g),
    \]
    and therefore if $(g,h) \in Y_i \cap Y_j$ for some $i < j$, then $h_iC_G(g) = h_jC_G(g)$, or equivalently $g \in C_G(h_i^{-1}h_j)$, which is impossible since $g \notin Y$.  Thus the subsets $Y_1,\ldots,Y_n$ are pairwise disjoint, implying that
    \[
    1 = \mu(G^2) \geq \mu(Y_1 \cup \cdots \cup Y_n) = \sum_{i=1}^n \mu(Y_i) \geq n\varepsilon,
    \]
    contradicting the fact that $n > \frac{1}{\varepsilon}$.
    
    Hence we must have $\mu(\G{n} \times G) > 0$ for some $n \in \N$, as required.
\end{proof}

\begin{prop} \label{prop:am-faf}
    Let $G$ be a FAF group.  Then there exists a constant $\alpha > 0$ such that $\mu(\Comm(G)) \geq \alpha$ for any invariant mean $\mu$ on $G^2$.
\end{prop}

\begin{proof}
    By Remark~\ref{rem:FCAF}, $G$ has a finite index subgroup $G_0 \leq G$ such that $G_0'$ is finite and central in $G_0$.  Then the function $\mu_0\colon \Pset(G_0^2) \to [0,1]$ defined by $\mu_0(A) = [G:G_0]^2 \mu(A)$ is an invariant mean on $G_0^2$, and if we have $\mu_0(\Comm(G_0)) > 0$ then it also follows that
    \[
    \mu(\Comm(G)) \geq \mu(\Comm(G_0)) = [G:G_0]^{-2}\mu_0(\Comm(G_0)) > 0.
    \]
    We may thus assume, by replacing $G$ with $G_0$ and multiplying $\alpha$ by $[G:G_0]^{-2}$ if necessary, that $G'$ is finite and central in $G$.  In particular, $G$ is $2$-step nilpotent.

    Now let $g_1 = [x_1,y_1],\ldots,g_n = [x_n,y_n]$ be all the possible commutators in $G$, and let $H = \bigcap_{i=1}^n (C_G(x_i) \cap C_G(y_i))$.  Note that we have $[G:C_G(x_i)],[G:C_G(y_i)] \leq |G'| < \infty$ for all $i$, and therefore $[G:H] < \infty$.  On the other hand, since $\mu$ is an invariant mean on $G^2$, we obtain
    \begin{align*}
        \mu(\Comm(G)) &= \frac{1}{n} \sum_{i=1}^n \mu\!\left( (x_i,y_i^{-1})\Comm(G) \right) \geq \frac{1}{n} \mu\!\left( \bigcup_{i=1}^n (x_i,y_i^{-1})\Comm(G) \right)
        \\ &\geq \frac{1}{n} \mu\!\left( H^2 \cap \bigcup_{i=1}^n (x_i,y_i^{-1})\Comm(G) \right) 
        \\ &= \frac{1}{n} \mu\!\left( \left\{ (x,y) \in H^2 : [x_i^{-1}x,y_iy] = 1 \text{ for some } i \in \{1,\ldots,n\} \right\} \right)
        \\ &= \frac{1}{n} \mu\!\left( \left\{ (x,y) \in H^2 : [x,y] = [x_i,y_i] \text{ for some } i \in \{1,\ldots,n\} \right\} \right),
    \end{align*}
    where the last equality follows from commutator calculus (in the $2$-step nilpotent group $G$) and the fact that $[x_i,y] = [x,y_i] = 1$ for all $(x,y) \in H^2$.  But since the $g_i = [x_i,y_i]$ are all the possible commutators in $G$, it follows that $\mu(\Comm(G)) \geq \frac{1}{n} \mu(H^2) = \frac{1}{n[G:H]^2} > 0$, so we obtain the result by setting $\alpha = \frac{1}{n[G:H]^2}$.
\end{proof}

\begin{thm} \label{thm:dc-amenable}
    \ThmDCAm
\end{thm}

\begin{proof}
    If $G$ is FAF, then $\mu(\Comm(G)) > 0$ by Proposition~\ref{prop:am-faf}.  Conversely, if $\mu(\Comm(G)) > 0$, then we have $\mu(\G{n} \times G) > 0$ for some $n \in \N$ by Corollary~\ref{cor:am-mu}.  Define $\mu'\colon \Pset(G) \to [0,1]$ by setting $\mu'(A) = \mu(A \times G)$ for all $A \subseteq G$, and note that $\mu'$ is clearly a left invariant finitely additive probability mean, and therefore a CCM, on $G$.  Since $\mu'(\G{n}) > 0$, it then follows from Proposition~\ref{prop:dc-faf} that $G$ is FAF.
\end{proof}

\begin{rem}
    A natural question is whether $\mu(\Comm(G))$ depends on the choice of an invariant mean $\mu$ on a FAF group $G$.  We have not been able to prove or to find a counterexample for this.  However, if a FAF group $G$ is residually finite (for instance, when it is finitely generated), then the independence of the choice of $\mu$ follows from the proof of Theorem~\ref{thm:dc-res-finite}.
\end{rem}

\subsection*{Residually amenable groups}

We may similarly deduce that if $G$ is a residually amenable group, then it has non-zero degree of commutativity if and only if it is is FAF.  In particular, for a residually amenable $G$, we can set $\dcRAm(G)$ to be the infimum of all $\mu(\Comm(H))$, where $H$ ranges over amenable quotients of $G$ and where $\mu$ ranges over invariant means on $H^2$.  If $G$ is FAF, then it is amenable (since finite and abelian groups are amenable and since amenability is closed under extensions), and so $\dcRAm(G)$ is equal to the infimum of all $\mu(\Comm(G))$, where $\mu$ ranges over invariant means on $G^2$, which is positive by Proposition~\ref{prop:am-faf}.

Conversely, if $\dcRAm(G) > 0$, then any amenable quotient of $G$ is FAF by Theorem~\ref{thm:dc-amenable}, and therefore virtually metabelian by Remark~\ref{rem:FCAF}.  Since $G$ is residually amenable, it follows that the residual $\res(G)$ of $G$ is residually metabelian, and therefore metabelian since being metabelian is defined by a group law.  On the other hand, since finite groups are amenable we have $\dcRF(G/\res(G)) \geq \dcRAm(G) > 0$, and so $G/\res(G)$ is virtually abelian: see Remark~\ref{rem:G/resG}.  It follows that $G$ is metabelian-by-(virtually abelian), and so virtually solvable of derived length~$3$; in particular, $G$ is amenable.  We therefore have $\mu(\Comm(G)) \geq \dcRAm(G) > 0$ for any invariant mean $\mu$ on $G^2$, and so $G$ is FAF by Theorem~\ref{thm:dc-amenable}.

\subsection*{Counterexamples}

Here we show that Theorem~\ref{thm:dc-amenable} does not hold for CCMs. That is, there exists a group $G$ and CCMs on $G^2$ such that the degrees of commutativity with respect to these CCMs on $G^2$ are 0 and 1.

\begin{ex} 
\label{ex:heisenberg}

    In this example we shall find a group $H$ and two CCMs $\mu_0, \mu_1$ on $H^2$ such that $\mu_0(\Comm(H)) =0$ whereas $\mu_1(\Comm(H))=1$. Here $H$ will be a nilpotent group, and therefore an amenable group. In fact, $H$ will be finite-by-abelian, FA (and hence also FAF). But $H$ is not abelian, and so one has that product means on $H^2$ give a degree of commutativity strictly between $0$ and $1$. The point being that neither $\mu_0$ nor $\mu_1$ are product means on the square, nor invariant means.

\medskip

    Let $\F_2$ be the field of order $2$, and consider the group
    \[
    H = \left\{ h(\mathbf{x},\mathbf{y},z) \coloneqq \begin{pmatrix} 1 & \mathbf{x} & z \\ 0 & 1 & \mathbf{y} \\ 0 & 0 & 1 \end{pmatrix} : \mathbf{x},\mathbf{y} \in V, z \in \F_2 \right\},
    \]
    where $V = \bigoplus_{i=1}^\infty \F_2$ and the group operation is ``matrix multiplication'':
    \[
    h(\mathbf{x},\mathbf{y},z) \cdot h(\mathbf{x}',\mathbf{y}',z') = h(\mathbf{x}+\mathbf{x}',\mathbf{y}+\mathbf{y}',z+z'+\langle \mathbf{x},\mathbf{y}' \rangle);
    \]
    here $\langle \mathbf{x},\mathbf{y}' \rangle = \sum_{i=1}^\infty x_iy_i'$ for $\mathbf{x} = (x_1,x_2,\ldots)$ and $\mathbf{y}' = (y_1',y_2',\ldots)$.  One may check that $Z(H) = \{ h(\mathbf{0},\mathbf{0},z) : z \in \F_2 \}$ and that $[H:C_H(h)] = 2$ for all $h \in H - Z(H)$.  We should thus expect to have ``$\dc(H) = \frac{1}{2}$'' for any ``reasonable'' definition of $\dc$. In particular, $\dcCCM(H) = \frac{1}{2}$. 
    
    However, we will show that there exists a CCM $\mu_1$ on $H^2$ for which $\mu_1(\Comm(H)) = 1$, by using the construction of CCMs in \S\ref{sec:CCMs}.  In order to do this, it is enough to show that the set $\Pfin(\Comm(H))$ generates a filter together with the sets of the form
    \[
    \mathcal{F}_\varepsilon(K) \coloneqq \left\{ F \in \Pfin(H^2) : \left| \frac{|gK \cap F|}{|F|} - \frac{1}{[H^2:K]} \right| < \varepsilon \text{ for all } g \in H^2 \right\}
    \]
    where $K \leq H^2$ and $\varepsilon > 0$.  In particular, let $K_1,\ldots,K_n \leq H^2$ and $\varepsilon > 0$.  It is then enough to show that $\bigcap_{i=1}^n \mathcal{F}_\varepsilon(K_i)$ contains an element $F \subset H^2$ such that $[h_1,h_2] = 1$ for all $(h_1,h_2) \in F$.

    Suppose, without loss of generality, that for some $m$ we have $[H^2:K_i] < \infty$ for $1 \leq i \leq m$ and $[H^2:K_j] = \infty$ for $m+1 \leq j \leq n$.  For $m+1 \leq j \leq n$, note that $K_j$ is a finite index subgroup of $K_jZ(H)^2$, and thus $\widehat{K}_j \coloneqq K_jZ(H)^2/Z(H)^2$ is an infinite index subgroup of $(H/Z(H))^2 \cong V^4 \cong (\bigoplus_{i=1}^\infty \F_2,+)$.  It follows that $\widehat{K}_j$ is an $\F_2$-subspace of $(H/Z(H))^2$ of infinite codimension, and hence is contained in an $\F_2$-subspace $\widehat{L}_j$ of finite codimension $N > -\log_2 \varepsilon$.  Letting $L_j$ be the preimage of $\widehat{L}_j$ under the quotient map $H^2 \to (H/Z(H))^2$, it then follows that we have $K_j < L_j \leq H^2$ and $\frac{1}{\varepsilon} < [H^2:L_j] < \infty$.

    Now let $K = \bigcap_{i=1}^m K_i \cap \bigcap_{j=m+1}^n L_j$.  Then $K$ is a finite index subgroup of $H^2$.  By Lemma~\ref{lem:heisenberg} below, there exists a (left) transversal $F$ of $K$ in $H^2$ such that $F \subset \Comm(H)$.  By construction, for any $g \in H^2$ we have
    \[
    \frac{|gK_i \cap F|}{|F|} = \frac{1}{[H^2:K_i]}
    \]
    for $1 \leq i \leq m$ and
    \[
    \left| \frac{|gK_j \cap F|}{|F|} - \frac{1}{[H^2:K_j]} \right| = \frac{|gK_j \cap F|}{|F|} \leq \frac{|gL_j \cap F|}{|F|} = \frac{1}{[H^2:L_j]} < \varepsilon
    \]
    for $m+1 \leq j \leq n$, implying that $F \in \bigcap_{i=1}^n \mathcal{F}_\varepsilon(K_i)$, as required.

\medskip

    Using a similar argument, we may show that $\Pfin(\Comm(H)^c)$ together with the $\mathcal{F}_\varepsilon(K)$ generates a filter, and therefore there exists another CCM $\mu_0$ on $H^2$ such that $\mu_0(\Comm(H)) = 0$.
\qed
\end{ex}

This is the technical Lemma used in the example above. 

\begin{lem} \label{lem:heisenberg}
    Let $H$ be as in Example~\ref{ex:heisenberg}, let $K \leq H^2$ be a finite index subgroup, and let $g \in H^2$.  Then there exists an element $(h_1,h_2) \in gK$ such that $[h_1,h_2] = 1$.
\end{lem}

\begin{proof}
    Since $K$ has finite index in $H^2$ we have $K \cap (H \times \{1\}) = L_1 \times \{1\}$ and $K \cap (\{1\} \times H) = \{1\} \times L_2$ for finite index subgroups $L_1,L_2 \sgpfin H$. Let $L = L_1 \cap L_2$.  Then $L^2 \sgpfin K \sgpfin H^2$ and so we may assume, without loss of generality, that $K = L^2$.

    Let $g = (g_1,g_2) \in H^2$, and suppose for contradiction that $[g_1h_1,g_2h_2] \neq 1$ for all $h_1,h_2 \in L$.  Note that, as $H$ is $2$-step nilpotent, we have
    \[
    [g_1h_1,h_2] = [g_1h_1,g_2]^{-1}[g_1h_1,g_2h_2],
    \]
    and thus, since $[H,H]$ is cyclic of order $2$, we have $[g_1h_1,h_2] = 1$.  Since $h_2 \in L$ was arbitrary, it follows that $L \subseteq C_H(g_1h_1)$ for all $h_1 \in L$.  Taking $h_1 = 1$ we get $L \subseteq C_H(g_1)$, implying that $L \subseteq C_H(g_1h) \cap C_H(g_1) \subseteq C_H(h)$ for all $h \in L$.  In particular, $L$ is abelian.

    On the other hand, $L$ has finite index in $H$, so this contradicts the fact that $H$ is not virtually abelian (see \cite[\S 7.3]{MTVV}).
\end{proof}

\begin{rem}
    The CCM $\mu$ in Example~\ref{ex:heisenberg} is not an invariant mean, as for any invariant mean $\mu$ on $G^2$, where $G$ is a non-abelian group, we have $\mu(\Comm(G)) < 1$.  Indeed, if $g \in G$ is any non-central element, then we have
    \begin{align*}
        \Comm(G)^c \cup (1,g)\Comm(G)^c &= \{ (x,y) \in G^2 : [x,y] \neq 1 \text{ or } [x,g^{-1}y] \neq 1\} \\ &= \{ (x,y) \in G^2 : [x,y] = [x,g^{-1}y] = 1\}^c \\ &\supseteq \{ (x,y) \in G^2 : [x,g]=1 \}^c = (G-C_G(g)) \times G,
    \end{align*}
    so if $\mu$ is an invariant mean on $G^2$ then we get
    \[
    \mu(\Comm(G)^c) \geq \frac{1}{2} \mu((G-C_G(g)) \times G) = \frac{1}{2}\left( 1 - \frac{1}{[G:C_G(g)]} \right) \geq \frac{1}{4}
    \]
    and therefore $\mu(\Comm(G)) \leq \frac{3}{4}$.
\end{rem}

\section{Means on non-amenable groups}
\label{sec:defect}

Since the paper has discussed various means on groups that exhibit partial invariance, it is natural to ask whether this can be quantified. We do this via a defect function which measures the worst case in which a mean on a group deviates from invariance. We show that while this is defined as a real number in the interval $[0,1]$, it turns out to be equal to either $0$ or $1$ with the former case corresponding to the case of amenability. Therefore, even though we construct means which are invariant on subgroups, the defect function sees these as generally no closer to an invariant mean than an arbitrary one.

\begin{defn}
\label{defn:defect}
    Given a finitely additive probability mean $\mu$ on a group $G$, we define the \emph{left defect} of $\mu$ as
    \[
    \delta_\ell(\mu) = \sup \{ |\mu(gA)-\mu(A)| : g \in G, A \subseteq G \} \in [0,1],
    \]
    and the \emph{right defect} of $\mu$ as
    \[
    \delta_r(\mu) = \sup \{ |\mu(Ag)-\mu(A)| : g \in G, A \subseteq G \} \in [0,1].
    \]
    Moreover, we define the (\emph{left}) \emph{defect} of $G$ to be
    \[ \delta_G = \inf_{\mu} \delta_\ell(\mu), \]  where the infimum ranges over all finitely additive probability means $\mu$ on $G$. 
\end{defn}

\begin{rem}
    We are omitting the right defect for a group but it has an analogous definition and is easily seen to be equal to the left defect. That is, if \(\mu\) is a finitely additive probability mean, we can define \(\nu(A) \coloneqq \mu(A^{-1})\), and then \(\delta_\ell(\nu) = \delta_r(\mu)\).
\end{rem}

\begin{lem} \label{lem:defects}
    Let $\mu$ be a finitely additive probability mean on a group $G$.
    \begin{enumerate}
        \item \label{it:defects-left} Suppose that $\delta_\ell(\mu) > 0$ and $\delta_r(\mu) < 1$.  Then $G$ admits a finitely additive probability mean $\overline\mu$ such that $\delta_\ell(\overline\mu) < \delta_\ell(\mu)$ and $\delta_r(\overline\mu) \leq \delta_r(\mu)$.
        \item \label{it:defects-right} Suppose that $\delta_r(\mu) > 0$ and $\delta_\ell(\mu) < 1$.  Then $G$ admits a finitely additive probability mean $\overline\mu$ such that $\delta_r(\overline\mu) < \delta_r(\mu)$ and $\delta_\ell(\overline\mu) \leq \delta_\ell(\mu)$.
    \end{enumerate}
\end{lem}

\begin{proof}
    Suppose first that $\delta_\ell(\mu) > 0$ and $\delta_r(\mu) < 1$. Define $\overline\mu\colon \Pset(G) \to [0,1]$ by setting
    \[
    \overline\mu(A) = \int \mu(hA) \,\mathrm{d}\mu(h)
    \]
    for all $A \subseteq G$.  It is then easy to see that $\overline\mu$ is a finitely additive probability mean on $G$ satisfying $\delta_r(\overline\mu) \leq \delta_r(\mu)$.  Now let $A \subseteq G$, let $m_- = \inf_{g \in G} \mu(gA)$, let $m_+ = \sup_{g \in G} \mu(gA)$, let $m = \frac{m_-+m_+}{2}$, let $D = \frac{m_+-m_-}{2}$, and let $\delta = \delta_\ell(\mu)$.  Note that $D = m_+-m = m-m_-$ and that $m_+-m_- \leq \delta$; in particular, $D \leq \frac{\delta}{2}$.  Setting $B = \{ h \in G : \mu(hA) \geq m \}$ we get $Bg^{-1} = \{ h \in G : \mu(hgA) \geq m \}$ for any $g \in G$, which allows us to estimate
    \[
    m_- + D\mu(Bg^{-1}) \leq \overline\mu(gA) \leq m + D\mu(Bg^{-1}).
    \]
    In particular, we have
    \begin{align*}
        \overline\mu(gA)-\overline\mu(A) &\leq \left( m+D\mu(Bg^{-1}) \right) - \left( m_-+D\mu(B) \right) = D(1+\mu(Bg^{-1})-\mu(B)) \\ &\leq \frac{\delta}{2} \left( 1+\delta_r(\mu) \right)
    \end{align*}
    for all $g \in G$ and $A \subseteq G$.  Since $\delta_r(\mu) < 1$ and $\delta > 0$, this implies that $\delta_\ell(\overline\mu) \leq \frac{1+\delta_r(\mu)}{2} \delta < \delta$, proving~\ref{it:defects-left}.

    Part~\ref{it:defects-right} follows in a similar way to part~\ref{it:defects-left}, or by applying part~\ref{it:defects-left} to the finitely additive probability mean $\mu'$ on $G$ defined by $\mu'(A) = \mu(A^{-1})$.
\end{proof}

\begin{prop}
\label{prop: defect less than 1 is amenable}
    Let $\nu$ be a finitely additive probability mean on a group $G$, and suppose that $\delta_\ell(\nu) < 1$.  Then $G$ is amenable.
\end{prop}

\begin{proof}
    Throughout this proof, all means on $G$ will be assumed to be finitely additive probability means.

    Let $\delta_G = \inf_{\nu'} \delta_\ell(\nu')$, where the infimum ranges over all means $\nu'$ on $G$.  Then there is a sequence $(\mu_n)_{n \in \N}$ of means on $G$ such that $\delta_G = \lim_{n \to \infty} \delta_\ell(\mu_n)$.  Let $\omega$ be a non-principal ultrafilter on $\N$, and let $\mu\colon \Pset(G) \to [0,1]$ be defined by setting $\mu(A) = \lim_\omega \mu_n(A)$.  It is then easy to verify that $\mu$ is a finitely additive probability mean on $G$ and that $|\mu(gA)-\mu(A)| \leq \delta_G$ for all $g \in G$ and $A \subseteq G$.  In particular, we have $\delta_\ell(\mu) = \delta_G$.  Note that if $\delta_G = 0$ (respectively $\delta_r(\mu) = 0$), then $\mu$ is a left invariant (respectively right invariant) mean on $G$ and so $G$ is amenable.
    
    Thus, suppose for contradiction that $\delta_G > 0$ and $\delta_r(\mu) > 0$; note that we also have $\delta_G \leq \delta_\ell(\mu) < 1$.  By Lemma~\ref{lem:defects}\ref{it:defects-right}, there exists a mean $\overline\mu$ on $G$ satisfying $\delta_\ell(\overline\mu) \leq \delta_G$ (and so $\delta_\ell(\overline\mu) = \delta_G > 0$ by the definition of $\delta_G$) and $\delta_r(\overline\mu) < \delta_r(\mu)$ (and so $\delta_r(\overline\mu) < 1$).  But then, by Lemma~\ref{lem:defects}\ref{it:defects-left} (applied to $\overline\mu$), it follows that there exists a mean $\overline{\overline\mu}$ on $G$ satisfying $\delta_\ell(\overline{\overline\mu}) < \delta_\ell(\overline\mu) = \delta_G$, contradicting the definition of $\delta_G$.
\end{proof}

\begin{thm} \label{thm:defect}
    \ThmDefect
\end{thm}

\begin{proof}
     Note that the proof of Proposition~\ref{prop: defect less than 1 is amenable} already shows that the defect of $G$ is realised on an actual mean, via ultralimits. However, we can also argue\textit{ a posteriori}. 
    
    If $G$ is amenable then it admits a left invariant mean $\mu$ whose defect is $0$, so $\delta_G=\delta_{\ell}(\mu) = 0$. If $G$ is not amenable then Proposition~\ref{prop: defect less than 1 is amenable} implies that $\delta_{\ell}(\mu) = 1$ for all means $\mu$ on $G$ and hence $\delta_G = 1$. 
\end{proof}

\section{Conjugacy classes}
\label{sec:conj-classes}

For a finite group, we have $\dc(G) = \cc(G)/|G| \eqqcolon k(G)$, where $\cc(G)$ is the number of conjugacy classes of $G$. For a residually finite group $G$, we similarly get $k_{\text{RF}}(G) \coloneqq \inf_{N \sgpnfin G} k(G/N) = \dcRF(G)$.  This motivates the following definition.

\begin{defn}
    Let $\mu$ be a CCM on a group $G$. Define
    \[
    k_\mu(G) \coloneqq \inf \{ \mu(S) : S \text{ is a set of representatives of conjugacy classes of } G \}.
    \]
\end{defn}

For the following result, note that any FAF group is amenable and therefore admits a finitely additive bi-invariant probability mean.

\begin{lem} \label{lem:cc-faf}
    Let $G$ be a FAF group equipped with a finitely additive bi-invariant probability mean $\mu$. Then $k_\mu(G) > 0$.
\end{lem}

\begin{proof}
    Let $H \sgpnfin G$ be such that $N \coloneqq [H,H]$ is finite.  Suppose that $S$ is a set of representatives of conjugacy classes of $G$, so that we have
    \[
    H \cap S = \{ h^{\varphi(h)} : h \in H \},
    \]
    for some function $\varphi\colon H \to G$.  Given any coset $A = gH \in G/H$, let $T_A = \{ h \in H : \varphi(h) \in A \}$.  Since $\mu$ is finitely additive, we have $\frac{1}{[G:H]} = \mu(H) = \sum_{A \in G/H} \mu(T_A)$; in particular, we have $\mu(T_{A_0}) \geq \frac{1}{[G:H]^2}$ for some $A_0 = g_0H \in G/H$.

    Given any $h \in T_{A_0}$, note that we have $g_0 = \varphi(h) \psi(h)$ for some $\psi(h) \in H$, and therefore
    \[
    h^{g_0} = h^{\varphi(h) \psi(h)} = h^{\varphi(h)} [h^{\varphi(h)},\psi(h)] \in h^{\varphi(h)}N,
    \]
    implying that $g_0^{-1} T_{A_0} g_0 \subseteq SN = \bigcup_{z \in N} Sz$.  Since $\mu$ is bi-invariant, we get $\mu(g_0^{-1}T_{A_0}g_0) = \mu(T_{A_0})$, and thus
    \[
    \mu(S) = \frac{1}{|N|} \sum_{z \in N} \mu(Sz) \geq \frac{1}{|N|} \mu\!\left( \bigcup_{z \in N} Sz \right) \geq \frac{1}{|N|} \mu(T_{A_0}) \geq \frac{1}{|N| \cdot [G:H]^2},
    \]
    implying that $k_\mu(G) \geq \frac{1}{|N| \cdot [G:H]^2} > 0$, as required.
\end{proof}

The following example shows that the assumption that $\mu$ is bi-invariant, and not merely left invariant, is necessary in Lemma~\ref{lem:cc-faf}.  By inspecting the argument, however, one may deduce that an assumption that $\mu$ is conjugacy invariant and left invariant on cosets would be enough.

\begin{ex}
    Let $G = \Z^2 \rtimes (\Z/4\Z)$, where the generator $t$ of $\Z/4\Z$ acts on $H = \Z^2$ as the matrix $\begin{pmatrix} 0 & 1 \\ -1 & 0 \end{pmatrix}$. Then $G$ has a set $S$ of representatives of conjugacy classes of $G$ such that $S \cap H \subseteq \{ (x,y) \in \Z^2 : |y| \leq x \}$.  Moreover, one can verify that $G$ contains only finitely many conjugacy classes outside of $H$, and therefore $\mu(S) = \mu(S \cap H)$ for any CCM $\mu$ on $G$.

    Now let $\mu_1$ be a finitely additive left invariant probability mean on $\Z$, and define $\mu_2\colon \Pset(\Z^2) \to [0,1]$ by setting
    \[
    \mu_2(A) = \int \mu_1(A_x) \,\mathrm{d}\mu_1(x),
    \]
    where $A_x = \{ y \in \Z : (x,y) \in A \}$.  It is easy to see that $\mu_2$ is a finitely additive left invariant probability mean on $\Z^2$.  Finally, define $\mu\colon \Pset(G) \to [0,1]$ by setting
    \[
    \mu(A) = \frac{1}{4} \sum_{i=0}^3 \mu_2(t^iA \cap H),
    \]
    which is again a finitely additive left invariant probability mean.  We then get
    \[
    \mu(S) = \mu(S \cap H) = \frac{1}{4} \mu_2(S \cap H) \leq \frac{1}{4} \int_{x \in \N} \mu_1(\{-x,-x+1,\ldots,x\}) \,\mathrm{d}\mu_1(x) = \frac{1}{4} \int 0 \,\mathrm{d}\mu_1 = 0,
    \]
    even though $G$ is virtually abelian.
\end{ex}

\begin{prop}
    Let $G$ be a group equipped with a CCM $\mu$, and suppose that $k_\mu(G) > 0$. Then $G$ is FAF.
\end{prop}

\begin{proof}
    Let $n \in \N$ be such that $k_\mu(G) > \frac{1}{n+1}$.  We claim that then we have $\mu(\G{n}) > 0$, where $\G{n} = \{ g \in G : [G:C_G(g)] \leq n \}$.  The result will then follow from Proposition~\ref{prop:dc-faf}.

    Indeed, we can choose representatives $S_0,\ldots,S_n$ of conjugacy classes of $G$ as follows.  For $g \in \G{n}$, choose the representative of $[g]$ in each $S_i$ arbitrarily.  For $g \in G - \G{n}$, note that the conjugacy class of $g$ contains $[G:C_G(g)] \geq n+1$ elements, and so we can choose $n+1$ different representatives of $[g]$ in the sets $S_0,\ldots,S_n$.  Then each element of $\G{n}$ appears in the sets $S_0,\ldots,S_n$ at most $n+1$ times, whereas each element of $G - \G{n}$ appears there at most once.  Since $\mu$ is finitely additive, we then have
    \[
    1 < (n+1) k_\mu(G) \leq \sum_{i=0}^n \mu(S_i) \leq \mu(G - \G{n}) + (n+1)\mu(\G{n}) = 1 + n \cdot \mu(\G{n}),
    \]
    and thus $\mu(\G{n}) > 0$, as claimed.
\end{proof}

\appendix

\section{Using the Hahn--Banach Theorem to construct CCMs} \label{app:Hahn-Banach}

In this appendix we give another proof of the existence of CCMs for any group. The reason to do so is that in Theorem~\ref{thm: existence of CCMs} we invoke the Ultrafilter Lemma which is a weak version of the Axiom of Choice. However, one can also prove the existence of CCMs by using the Hahn--Banach Theorem which is strictly weaker than the Ultrafilter Lemma. 

\begin{thm}[Hahn--Banach Theorem for real normed vector spaces] \label{thm:hahn-banach}
    Let $(V, \|{\cdot}\|)$ be a real normed vector space. Suppose that $S$ is a subspace of $V$ and $f_S\colon S \to \R$ a bounded linear functional on~$S$. Then there exists an extension $f$ of $f_S$ to $V$ such that $f|_S=f_S$ and $\| f \|_{V^*} = \| f_S \|_{S^*} $. 
\end{thm}

\begin{rem}
    Recall that if $(V, \|{\cdot}\|)$ is a normed vector space with $V \neq \{ 0 \} $  and $f\colon V \to \R$ a linear functional then $\| f \|_{V^*} \coloneqq \sup_{0 \neq x \in V} \frac{|f(x)|}{\| x \|} $. 
\end{rem}

\begin{thm} \label{thm:HB implies CCM}
    The Hahn--Banach Theorem for real normed vector spaces implies the existence of a CCM for any group $G$. 
\end{thm}
\begin{proof}
    It is sufficient to prove the existence of a linear functional, $m\colon \ell^{\infty}(G) \to \R$, of norm~$1$ such that $m(\mathbbm{1}_{gH}) = \frac{1}{[G:H]}$ for all cosets $gH$ in $G$. Indeed, if this is the case then setting $\mu(A) = m(\mathbbm{1}_A)$ for any subset $A \subseteq G$ defines a CCM $\mu$ on $G$. To see that this defines a CCM observe that $\mu(G) = m(\mathbbm{1}_G)=1$ by definition, the finite additivity of $\mu$ is a consequence of the linearity of $m$, and positivity of $\mu$ follows from the fact that for all $A \subsetneq G$,  
    \[
    m(\mathbbm{1}_A) = m(\mathbbm{1}_G) - m(\mathbbm{1}_{A^c})= 1 - m(\mathbbm{1}_{A^c}) \geq 1 - \|m\|_{\ell^{\infty}(G)^*} \| \mathbbm{1}_{A^c} \|_{\infty} = 0. 
    \]
    Therefore, $\mu(A) \geq 0$ for all $A \subseteq G$. 

    Recall that $\cosets \subseteq \Pset(G)$ denotes the Boolean subring generated by cosets of subgroups of $G$, Definition~\ref{defn:cosets}. By Proposition~\ref{prop:CCM on cosets}, there exists a unique CCM $\mu$ on $\cosets$.

    Now let $\mathfrak{C} \leq \ell^\infty(G)$ be the subspace spanned by $\{ \mathbbm{1}_S : S \in \cosets \}$.  Note that, since $\cosets$ is a subring of $\Pset(G)$, any element of $\mathfrak{C}$ can be expressed as $\sum_{i=1}^n a_i \cdot \mathbbm{1}_{S_i}$, where $a_1,\ldots,a_n \in \R$ and where $S_1,\ldots,S_n \in \cosets$ form a partition of $G$.  We then define a functional $m_{\mathfrak{C}}\colon \mathfrak{C} \to \R$ by setting
    \[
    m_{\mathfrak{C}}\!\left( \sum_{i=1}^n a_i \cdot \mathbbm{1}_{S_i} \right) \coloneqq \sum_{i=1}^n a_i \cdot \mu(S_i).
    \]
    
    To see that $m_{\mathfrak{C}}$ is well-defined, note that if $\sum_{i=1}^n a_i \cdot \mathbbm{1}_{S_i} = \sum_{j=1}^m b_j \cdot \mathbbm{1}_{T_j}$ are two different expressions of an element $f \in \mathfrak{C}$ as above, then we have
    \[
    \sum_{i=1}^n a_i \cdot \mathbbm{1}_{S_i} = \sum_{i=1}^n \sum_{j=1}^m c_{ij} \cdot \mathbbm{1}_{U_{ij}} = \sum_{j=1}^m b_j \cdot \mathbbm{1}_{T_j},
    \]
    where $U_{ij} = S_i \cap T_j$ and where $a_i = c_{ij} = b_j$ whenever $U_{ij} \neq \varnothing$.  It then follows by the finite additivity of $\mu$ that
    \[
    \sum_{i=1}^n a_i \cdot \mu(S_i) = \sum_{i=1}^n \left( a_i \sum_{j=1}^m \mu(U_{ij}) \right) = \sum_{i=1}^n \sum_{j=1}^m c_{ij} \cdot \mu(U_{ij}) = \sum_{j=1}^m \left( b_j \sum_{i=1}^n \mu(U_{ij}) \right) = \sum_{j=1}^m b_j \cdot \mu(T_j),
    \]
    as required.  A similar argument shows that $m_{\mathfrak{C}}$ is linear.
    
    Finally, $m_{\mathfrak{C}}$ is bounded since for any $f = \sum_{i=1}^n a_i \cdot \mathbbm{1}_{S_i} \in \mathfrak{C}$ as above we have,
    \[
    |m_{\mathfrak{C}}(f)| \leq \sum_{i=1}^n |a_i| \cdot \mu(S_i) \leq \max_{1 \leq i \leq n} |a_i| \cdot \sum_{i=1}^n \mu(S_i) = \max_{1 \leq i \leq n} |a_i| \cdot \mu(G) = \max_{1 \leq i \leq n} |a_i| = \|f\|_{\infty}.
    \]
    Together with the fact that $m_{\mathfrak{C}}(\mathbbm{1}_G)= \| \mathbbm{1}_G \|_{\infty} =1$, this forces the norm of $m_{\mathfrak{C}}$ to be $1$.  

    By Theorem~\ref{thm:hahn-banach}, the functional $m_{\mathfrak{C}}$ then extends to a linear functional $m\colon \ell^\infty(G) \to \R$ which satisfies the required properties.
\end{proof}

\begin{rem}
    We note that Theorem~\ref{thm:neumann} from \cite{Neumann1954} is used in Proposition~\ref{prop:CCM on cosets} and so is an ingredient in both the construction using ultrafilters as well as this one using the Hahn--Banach Theorem. 
\end{rem}

\section{Constructing CCMs using random walks} \label{app:random-walks}

We note that our notion of a CCM is equivalent to the notion of a sequence of probability measures which measure index uniformly, as in \cite{Tointon2020}. In that paper it is shown---in different language---that all finitely generated groups admit a CCM. We shall briefly show how to conclude that all groups admit a CCM from this result. 

\begin{defn}[{\cite[Definition~1.7]{Tointon2020}}]
 We say that a sequence $M = (\mu_n)_{n=1}^{\infty}$ of probability measures on a group $G$ \emph{measures index uniformly} if $\mu_n(xH) \to \frac{1}{[G:H]}$ uniformly over all $x \in G$ and all subgroups $H$ of $G$. 
\end{defn}

\begin{rem}
\label{rem:sequenceareCCM}
    If we are given a sequence of probability measures that measure index uniformly, then this produces a CCM by taking an ultralimit with respect to any non-principal ultrafilter on the natural numbers. Conversely, any CCM produces a sequence of probability measures which measure index uniformly by taking  the constant sequence. 
\end{rem}

We note that \cite{Tointon2020} shows the following, which therefore implies that all finitely generated groups admit a CCM. 

\begin{thm}[random walks measure index uniformly {\cite[Theorem~1.11]{Tointon2020}}] \label{thm:random-walks}
Let \( G \) be a finitely generated group, and let \( \mu \) be a symmetric, finitely supported, generating probability measure on \( G \) such that \( \mu(1) > 0 \). Then the sequence \( M_\mu = (\mu^{*n})_{n=1}^\infty \) measures index uniformly.
\end{thm}

\begin{cor}
    Every finitely generated group admits a CCM. 
\end{cor}

We now argue how this can be used to give another proof of the existence of CCMs for arbitrary groups.

\begin{thm} \label{thm:CCMsviawalks}
    Suppose $G$ is a group, each of whose finitely generated subgroups admits a CCM. Then $G$ also admits a CCM. 
\end{thm}
\begin{proof}
    First consider $X$ to be the set of all finitely generated subgroups of $G$. We note that this is a directed poset under inclusion: if $H, K$ are finitely generated subgroups of $G$, then $\langle H, K \rangle$ is another finitely generated subgroup of $G$ containing both of them. 

    We may thus define, for any $H \in X$, the tail of $H$ to be: 
    \[
    \operatorname{Tail}(H) \coloneqq \{ K \in X  :  H \leq K  \}.
    \]
    We note that the fact that $X$ is directed implies that any finite intersection of tails is non-empty. In particular, the tails generate a filter and hence there is an ultrafilter $\mathcal{U}$ containing all the tails, by  Proposition~\ref{prop:filter generation} and Theorem~\ref{thm:UL}. 

    We can therefore define a CCM on $G$ as follows:
    \[
    \mu(A) \coloneqq \lim_{\mathcal{U}} \mu_H(A \cap H) \quad \text{for } A \subseteq G \text{ and } H \in X,
    \]
    where $\mu_H$ is a CCM on the finitely generated subgroup, $H$, which exists by hypothesis. (Formally, we are thinking of a function from $X$ to $\R$ which sends $H$ to $\mu_H(A \cap H)$, and taking the ultralimit of that with respect to $\mathcal{U}$.) 

    It is clear that this is a finitely additive probability mean. All that remains is to check that it is constant on the cosets of any given subgroup. To that end, suppose $L$ is an \textit{arbitrary } subgroup of $G$ and $x \in G$.

    Now consider some $H \in \operatorname{Tail}(\langle x \rangle)$; that is, a finitely generated subgroup of $G$ such that $x \in H$. Then one easily verifies that $x L \cap H = x ( L \cap H)$ and hence $\mu_H(xL \cap H) = \mu_H(x (L \cap H)) = \mu_H(L \cap H)$, since $L \cap H$ is a subgroup of $H$. Hence we get that $\mu_H(xL \cap H) =  \mu_H(L \cap H)$ for all $H$ which lie in $\operatorname{Tail}(\langle x \rangle)$. In particular, since these subsets are elements of $\mathcal{U}$ by construction, $\mu$ is invariant on cosets of $L$ for any $L$ and so $\mu$ is a CCM. 
\end{proof}

\begin{rem}
    The proof presented above uses the full power of the Axiom of Choice, by choosing a CCM on each finitely generated subgroup of $G$.  However, it is easy to modify the proof so that it only uses the Ultrafilter Lemma.  Indeed, instead of taking the poset $X$ of all finitely generated subgroups of $G$, one may take the poset $X'$ of all finite subsets of $G$ containing the identity, which is again a directed set.  Then one may use Theorem~\ref{thm:random-walks} to construct a sequence $M_S = (\nu_S^{*n})_{n=1}^\infty$ that measures index uniformly on $\langle S \rangle$ for any $S \in X'$, where $\nu_S$ is the uniform measure on $S$.
    
    The rest of the argument follows as in the proof of Theorem~\ref{thm:CCMsviawalks} above, by replacing $\mathcal{U}$ by an ultrafilter on $X'$, and replacing $\mu_H$ by $\lim_\omega \nu_S^{*n}$ for $H = \langle S \rangle$, where $S \in X'$ and $H \in X$, and where $\omega$ is a (fixed) non-principal ultrafilter on the natural numbers.  This only requires the choice of two non-principal ultrafilters---one on $X'$ and one on $\mathbb{N}$---which can be made by only using the Ultrafilter Lemma.
\end{rem}

\bibliographystyle{amsalpha}
\bibliography{ref}

\end{document}